\documentclass[12pt]{amsart}

\usepackage{DKstyle}
\usepackage{bm}
\usepackage{float}
\usepackage{xcolor}
\usepackage[font=small,labelfont=bf]{caption}

\newcommand{\vertiii}[1]{{\left\vert\kern-0.25ex\left\vert\kern-0.25ex\left\vert #1
    \right\vert\kern-0.25ex\right\vert\kern-0.25ex\right\vert}}
    \makeatletter
\newcommand*{\rom}[1]{\expandafter\@slowromancap\romannumeral #1@}
\makeatother
\theoremstyle{plain}

\newtheorem*{theorem*}{Theorem}
\newtheorem{innercustomthm}{Theorem}
\newenvironment{customthm}[1]
  {\renewcommand\theinnercustomthm{#1}\innercustomthm}
  {\endinnercustomthm}

\newcommand{\area}{\operatorname{area}}

\usepackage{caption}
\usepackage[labelfont=rm]{subcaption}

\usepackage[pagebackref=true, colorlinks]{hyperref}

\hypersetup{pdffitwindow=true,linkcolor=blue,citecolor=blue,urlcolor=blue,menucolor=blue}

\usepackage{comment}


\date{\today}%
\dedicatory{}%
\commby{}%

\newcommand{\ssm}{\smallsetminus}
\title{A dynamical Borel-Cantelli lemma \\ via improvements to Dirichlet's theorem}
\author{Dmitry Kleinbock}
\address{Brandeis University, Waltham MA, USA, 02454-9110}
\email{kleinboc@brandeis.edu} 

\author{Shucheng Yu}
\address{Department of Mathematics, Technion, Haifa, Israel}
\email{yushucheng@campus.technion.ac.il}
\thanks{D.K.\ has been supported by NSF grants  DMS-1600814 and DMS-1900560. S.Y.\ acknowledges that this project has received funding from the European Research Council (ERC) under the European Union's Horizon 2020 research and innovation program (grant agreement No.\ 754475).}
\begin{document}

\begin{abstract} 
Let $X\cong \SL_2(\R)/\SL_2(\Z)$ be the space 
%{how about using ``space'' instead?} 
of unimodular lattices in $\R^2$, and for any $r\ge 0$ denote by $K_r\subset X$ the set of lattices such that all its nonzero vectors have supremum norm at least $e^{-r}$. These are compact nested subset{s} of $X$, with $K_0 = {\bigcap}_{r}K_r$ being the union of two closed horocycles. We use {an} explicit second moment formula for the Siegel transform of the indicator functions of squares in $\R^2$ centered at the origin to derive an asymptotic formula for the volume of sets $K_r$ as $r\to 0$. Combined with a zero-one law for the set of the $\psi$-Dirichlet numbers established by Kleinbock and Wadleigh \cite{KleinbockWadleigh2018}, this gives a new dynamical Borel-Cantelli lemma for the geodesic flow on $X$ with respect to the family of shrinking targets $\{K_r\}$.
% with the shrinking targets level sets of the shortest vector function
%We prove a dynamical Borel-Cantelli lemma for the geodesic flow on the space of rank two unimodular lattices with the shrinking targets level sets of the shortest vector function. Our proof replies on a zero-one law for the set of the $\psi$-Dirichlet numbers established in \cite{KleinbockWadleigh2018} and a completely explicit second moment formula for the Siegel transform of the indicator functions of small squares in the plane.
\end{abstract}
\maketitle
\section{Introduction}
%Let $k\geq 2$ be a positive integer and let $X=\SL_k(\R)/\SL_k(\Z)$ be the space of unimodular lattice in $\R^k$. 
Let $(X,\mu)$ be a probability space,
%$X=G/\G$ be a homogeneous space,
% endowed with a probability measure $\mu$ coming from the Haar measure of $G$, 
%where $G$ is a %connected semisimple 
%Lie group %without compact factors, 
%and $\G<G$ is %an irreducible 
%a lattice in $G$. Let $\mu$ be the probability measure on $X$ coming from a Haar measure {on} $G$ 
and let $\{a_s\}_{s\in\R}
%\subset G
$ be 
%an unbounded one-parameter subgroup which we view as 
a one-parameter measure-preserving  flow on $X$. Given a family of measurable subsets $\{B_s\}_{s>0}$ of $X$ with $\mu(B_s)\to 0$ as $s\to \infty$ (called \textit{shrinking targets}), the \textit{shrinking targets problem} %s the dynamical problem asking f
asks for a dichotomy on whether generic orbits of $\{a_s\}_{s>0}$ would hit the shrinking targets indefinitely. % the events $g_tx\in B_t$ happen for unbounded many $t$. 
That is, we are looking for a zero-one law for the measure of the limsup set 
$$B_{\infty}:=\limsup_{s\to\infty}a_{-s}B_s=\left\{x\in X\ \left|\ a_sx\in B_s\ \textrm{for {an unbounded set of} $s>0$}\right.\right\}.$$
For any $n\in \N$ let 
\begin{equation}\label{equ:thickenings}
\widetilde{B}_n:=\bigcup_{0\leq s<1}a_{-s}B_{n+s}\end{equation}
 be the thickening of the shrinking targets $\{B_s\}_{n\leq s<n+1}$ along the flow $\{{a_{-s}}\}_{0\leq s<1}$. Note that $a_nx\in \widetilde{B}_n$ if and only if there exists some %$n\leq s<n+1$  
$s\in [n, n+1)$ such that $a_sx\in B_s$. We thus have 
\begin{equation}\label{equ:discrete}B_{\infty}=\limsup_{n\to\infty}a_{-n}\widetilde{B}_n=\big\{x\in X\ \bigm|\ a_nx\in \widetilde{B}_n\ \textrm{infinitely often}\big\},\end{equation} 
and the classical Borel-Cantelli lemma implies that 
\begin{equation}\label{equ:easy}
\sum_{n}\mu(\widetilde{B}_n)<\infty\quad\Longrightarrow\quad \mu(B_{\infty})=0.\end{equation} 
 On the other hand, {following the terminology of \cite{ChernovKleinbock01}} we say the family of shrinking targets $\{B_s\}_{s>0}$ is \textit{Borel-Cantelli (BC)} for the flow $\{a_s\}_{s > 0}$ if $\mu(B_{\infty})=1$. Thus a necessary condition for $\{B_s\}_{s>0}$ to be BC for $\{a_s\}_{s > 0}$ is that the sequence of its thickenings has divergent sum of measures, {and we say $\{B_s\}_{s>0}$ satisfies \textit{a dynamical Borel-Cantelli lemma for $\{a_s\}_{s > 0}$} if this is also a sufficient condition}.

The shrinking targets problem for continuous time flow in the context of homogeneous spaces was first studied by Sullivan in \cite{Sullivan1982}, where he established a logarithm law for the fastest rate of geodesic cusp excursions in finite-volume hyperbolic manifolds. Later using the exponential mixing rate and a smooth approximation argument, {the first-named author} and Margulis \cite{KleinbockMargulis1999} proved that the family of cusp neighborhoods $\{\Phi^{-1}(r(s),\infty)\}_{s>0}$ with divergent sum of measures is BC for any diagonalizable flow on $(G/\G,\mu)$,
% endowed with a probability measure $\mu$ coming from the Haar measure of $G$, 
where $G$ is a connected semisimple Lie group without compact factors, $\G<G$ is an irreducible 
lattice, and $\mu$ is the probability measure on $X = G/\G$ coming from a Haar measure {on} $G$.
%a general homogeneous space $X$, 
Here $\Phi$ is a distance-like function on $X$ \cite[Definition 1.6]{KleinbockMargulis1999} and $r(\cdot)$ is a quasi-increasing function \cite[Section 2.4]{KleinbockMargulis1999}. Later Maucourant \cite{Maucourant06} obtained a similar dynamical Borel-Cantelli lemma for geodesic flows making excursions into shrinking hyperbolic balls (with a fixed center) on a finite-volume hyperbolic manifold. See \cite{Athreya09} for a survey on shrinking targets problems in dynamical systems.

One main reason that such dynamical Borel-Cantelli lemmas have gained much attention is due to their connections to metric number theory which {were} first explored by Sullivan in \cite{Sullivan1982}. Such connections were made more apparent later in \cite{KleinbockMargulis1999}. Let $m, l$ be two positive integers and let $M_{m,l}(\R)$ be the space of $m$ by $l$ real matrices. Given $\psi:[t_0,\infty)\to (0,\infty)$ a continuous non-increasing function, let us define $W(\psi)\subset M_{m,l}(\R)$, the set of $\psi$-approximable $m\times l$ real matrices such that $A\in W(\psi)$ if and only if there are infinitely many $\bm{q}\in \Z^{{l}}$ satisfying
\begin{equation*}\label{equ:approx}
\|A\bm{q}-\bm{p}\|^m<\psi\left(\|\bm{q}\|^{l}\right)\quad \textrm{for some $\bm{p}\in \Z^m$},
\end{equation*}
where $\|\cdot\|$ is the supremum norm on respective Euclidean spaces. The classical Khinchin-Groshev theorem gives an exact criterion on when $W(\psi)$ has full or zero  Lebesgue measure.
\begin{customthm}{KG}[Khinchin-Groshev]\label{KG}  
Given a continuous non-increasing $\psi$, %then 
$W(\psi)$ has full (resp.\ zero) Lebesgue measure if and only if the series $\sum_k\psi(k)$ diverges (resp.\ converges).
\end{customthm}
%\begin{theorem*}\label{thm:kg}
%Given a continuous non-increasing $\psi$, %then 
%$W(\psi)$ has full (resp.\ zero) Lebesgue measure if and only if the series $\sum_k\psi(k)$ diverges (resp.\ converges).
%\end{theorem*}
{See \cite{Schmidt80} for more details.} On the other hand, let $X=\SL_{m+l}(\R)/\SL_{m+l}(\Z)$ be the space of unimodular lattices in $\R^{m+l}$ and let $\Delta: X\to [0,\infty)$ be the function on $X$ %such that for any $\Lambda\in X$, \{defined by}
\begin{equation}\label{delta}\Delta(\Lambda):=\sup_{\bm{v}\in \Lambda\ssm \{\bm{0}\}}\log\left(\frac{1}{\|\bm{v}\|}\right).\end{equation}
Note that $\Delta(\Lambda) \ge 0$ for any $\Lambda\in X$ due to Minkowski's Convex Body Theorem, and  for all $r\ge 0$ the sets 
\begin{equation}\label{kr}{K_r:= {\Delta^{-1}([0,r])}}\end{equation}
(of lattices such that all its nonzero vectors have supremum norm at least $e^{-r}$) are compact due to Mahler's Compactness Criterion,  see e.g.\ \cite{Cassels1997}. %{Maybe it will be good to have a reference here, say to Cassels' Geometry of Numbers.} 
Following ideas of Dani \cite{Dani1985}, it was shown in \cite{KleinbockMargulis1999} that there exists a unique function $r=r_{\psi}: [s_0,\infty)\to \R$  depending on $\psi$ 
(this was referred to as the Dani Correspondence) such that $A\in M_{m,l}(\R)$ is $\psi$-approximable if and only if the events $a_s\Lambda_A\in \Delta^{-1}  \big(r(s),\infty\big)  $ happen for {an unbounded set of $s>s_0$}, where 
$$a_s=\diag (e^{s/m},\dots,e^{s/m},e^{-s/l},\dots, e^{-s/l})$$
{with $m$ copies of $e^{s/m}$ and $l$ copies of $e^{-s/l}$,} and $\Lambda_{A}=\begin{pmatrix}
I_m & A\\
 0 & I_l\end{pmatrix}\Z^{m+l}\in X$. %Let $\psi:[s_0,\infty)\to (0,\infty)$ be a continuous non-increasing function. A $m\times n$ real matrix $A$ is called $\psi$-approximable if the system of inequalities 
%
%has solutions in $(p,q)$ for infinitely many $s$. 
%Moreover, {via} a reduction argument 
This way {the first-named author} and Margulis showed 
%the Khinchin-Groshev t
Theorem \ref{KG} to be equivalent to a dynamical Borel-Cantelli lemma for the $a_s$-orbits making excursions into the %level sets
cusp neighborhoods $\Delta^{-1} \big(r(s),\infty\big)_{s>{s_0}}$,  and used this to give an alternative dynamical proof of Theorem \ref{KG} based on mixing properties of the $a_s$-action on $X$, see \cite{KleinbockMargulis1999, KleinbockMargulis2018}.

\medskip
More recently, for a given $\psi$ as above, the first-named author and Wadleigh \cite{KleinbockWadleigh2018} studied the finer problem of improvements to Dirichlet's Theorem. {See \cite{DavenportSchmidt70a,DavenportSchmidt70b} for the history of the problem of improving Dirichlet's Theorem.} {Following the definition in \cite{KleinbockWadleigh2018}} an $m$ by $l$ real matrix $A$ is called \textit{$\psi$-Dirichlet} if the system of inequalities 
$$\|A\bm{q}-\bm{p}\|^m<\psi(t)\quad \textrm{and}\quad \|\bm{q}\|^l<t$$
has solutions in $(\bm{p},\bm{q})\in \Z^m\times (\Z^l\ssm \{\bm{0}\})$ for all sufficiently large $t$. %Let $D(\psi)\subset M_{m,n}(\R)$ be the set of $\psi$-Dirichlet matrices. 
Following the general scheme developed in \cite{KleinbockMargulis1999} they gave a dynamical interpretation of $\psi$-Dirichlet matrices. Namely, they showed that {$A\in M_{m,l}(\R)$} is not $\psi$-Dirichlet if and only if the events $$a_s\Lambda_A\in K_{r(s)}$$ happen for {an unbounded set of $s>s_0$}, where $a_s,\Lambda_A$ and $r=r_{\psi}$ are all as above. Hence in this case the family of shrinking targets is given by $\{K_{r(s)}\}_{s>{s_0}}$, and one is naturally interested in whether this family of shrinking targets is {BC} for the flow $\{a_s\}_{s > 0}$.

However this dynamical interpretation is not helpful when it comes to determining necessary and sufficient condition{s} on $\psi$ guaranteeing that almost every (almost no) $A$ is 
$\psi$-Dirichlet. One of the main difficulties is that the shrinking targets $K_{r(s)}$ are far away from being ${\SO_{m+l}(\R)}$-invariant, and thus when applying the mixing properties of the $a_s$-action it will involve certain Sobolev norms which are hard to control. Still, using a different method based on continued fractions the aforementioned conditions were found in \cite{KleinbockWadleigh2018} for the case $m=l=1$. Namely, the following was proved:
\begin{customthm}{KW}[Kleinbock-Wadleigh]\label{KW}  
Let $\psi: [t_0,\infty)\to {(0,\infty)}$ be a continuous, non-increasing function satisfying 
\begin{equation}\label{equ:con1psi}
\textrm{the function $t\mapsto t\psi(t)$ is non-decreasing}
\end{equation} 
and
\begin{equation}\label{equ:con2psi}
t\psi(t)<1\quad \textrm{for all $t\geq t_0$}.
\end{equation}
Then if the series
\begin{equation}\label{equ:zeroone}
\sum_{n}\frac{-\left(1-n\psi(n)\right)\log\left(1-n\psi(n)\right)}{n} %=\infty\ (resp. <\infty),
\end{equation}
diverges (resp.\ converges), then  Lebesgue-a.e. $x\in\R$ is not (resp.\ is) $\psi$-Dirichlet.
\end{customthm}

In this paper we use the above theorem to derive a dynamical Borel-Cantelli lemma for the diagonal flow $a_s:=\diag (e^s,e^{-s})$ on $X:=\SL_2(\R)/\SL_2(\Z)$.
 Let %$X=\SL_2(\R)/\SL_2(\Z)$ and let 
$\mu$ be the probability Haar measure on $X$,
%. Let $a_s=\diag (e^s,e^{-s})$ be the diagonal flow on $X$ 
consider the function $\Delta$  on $X$ as in \eqref{delta}, and define the sets $K_r$  as in  \eqref{kr}. 
%For any continuous and non-increasing function $r: [s_0,\infty)\to (0,\infty)$ with $s_0>1$, we will take $B_s=K_{r(s)}$ as the family of shrinking targets.

\smallskip
We now state our dynamical Borel-Cantelli lemma.
\begin{thm}\label{thm:dynamical01}
Let $r:[s_0,\infty)\to (0,\infty)$ be a continuous {and} non-increasing function. Let $B_s=K_{r(s)}$ and let ${B_{\infty}}=\limsup_{t\to\infty}a_{-s}B_s$. Then we have 
$$\sum_nr(n)\log\left(\frac{1}{r(n)}\right)<\infty\ \Longrightarrow\ \mu({B_{\infty}})=0.$$
If in addition we assume that the function $s\mapsto s+r(s)$ is non-decreasing, then we have
$$\sum_nr(n)\log\left(\frac{1}{r(n)}\right)=\infty\ \Longrightarrow\ \mu({B_{\infty}})=1.$$
\end{thm}

Comparing the statement of the above theorem with \eqref{equ:easy}, one can guess that it can be  approached by studying the thickenings
 \begin{equation}\label{equ:newthickenings}
\widetilde{B}_n=\bigcup_{0\leq s<1}a_{-s}B_{n+s} = \bigcup_{0\leq s<1}a_{-s}K_{r(n+s)}
\end{equation}
 as in \eqref{equ:thickenings}.
 %We will prove %the asymptotic measure formula 
%\eqref{equ:asyfor} 
We do it in several steps. In the beginning of \S\ref{thickening} we prove an asymptotic measure formula for the %level 
sets $K_{r}$ where $r$ is small:

\begin{thm}\label{thm:measure}
For any $0<r<\frac12\log 2$ we have
$$\mu\left(K_{r}\right)=\frac{4r^2\log\left(\frac{1}{r}\right)}{\zeta(2)}+O(r^2),$$
where $\zeta(2)=\frac{\pi^2}{6}$ is the value of the Riemann zeta function at $2$.
\end{thm}

Here and hereafter for two positive quantities $A$ and $B$, we will use the notation $A\ll B$  {or $A=O(B)$} to mean that there is a constant $c>0$ such that $A\leq cB$, and we will use subscripts to indicate the dependence of the constant on parameters. We will write $A\asymp B$ for $A\ll B\ll A$.

The next step is to use Theorem \ref{thm:measure} to estimate the measure of the thickening of $K_r$ along the flow $\{{a_{-s}}\}_{0\leq s<1}$ by bounding it from above and below by a finite union of $a_s$-translates of %the level sets 
$K_{r}$. This is also done in  \S\ref{thickening} and yields the following result:

\begin{thm}\label{thm:thickening}
For any $0<r<\log 1.01$ we have
\begin{equation*}\label{equ:asyfor}
\mu\left(\bigcup_{0\leq s<1}a_{-s}K_{r}\right)\asymp r\log\left(\frac{1}{r}\right).
\end{equation*}
\end{thm}

 The above asymptotic equality shows that the series appearing in Theorem \ref{thm:dynamical01} converges/diverges iff so does the series  $\sum_n\mu(\widetilde{B}_n)$, where  $\widetilde{B}_n$ is as in \eqref{equ:newthickenings}:
 
 \begin{cor}\label{cor:series}
Let $r:[s_0,\infty)\to (0,\infty)$ be a %continuous and 
non-increasing function, and let $\widetilde{B}_n$ be as in \eqref{equ:newthickenings}. Then we have
$$\sum_n\mu(\widetilde{B}_n)=\infty\ \textrm{ if and only if }\ \sum_nr(n)\log\left(\frac{1}{r(n)}\right)=\infty.$$
\end{cor}

Therefore, in view of \eqref{equ:discrete} and \eqref{equ:easy}, the convergence part of Theorem \ref{thm:dynamical01} is immediate from the Borel-Cantelli lemma. The divergence part however is trickier. Instead of using a dynamical approach as in \cite{KleinbockMargulis1999}, our proof  in \S\ref{sec:shrinkingtargets} is non-dynamical and relies 
on Theorem \ref{KW} and the Dani Correspondence.

 \medskip
 
%Indeed,} 
%our first result towards the proof of Theorem \ref{thm:thickening}
%this dynamical Borel-Cantelli lemma 
%is a measure estimate %formula 
%for {the} thickenings of the %level 
%sets $K_{r}$ along the flow $\{a_s\}%_{s\in\R}
%$ for small $r$. 
%We will prove %the asymptotic measure formula 
%\eqref{equ:asyfor} in two steps. First we prove an asymptotic measure formula for the %level 
%sets $K_{r}$. Then with this asymptotic measure formula, we estimate the measure of the thickening by bounding it from above and below by a finite union of $a_s$-translates of the level sets $K_{r}$ respectively.

%To compute $\mu(K_{r})$, i
It remains to comment on our proof of Theorem \ref{thm:measure}. Instead of trying to describe the sets %this level set 
$K_{r}$ explicitly in terms of coordinates and compute their measures directly, we adapt an indirect approach which relies %the following completely
on an explicit second moment formula of the Siegel transform of certain indicator functions. Recall that if $f$ is a function on $\R^2$, its \textit{primitive Siegel transform} is the function on $X$ given by 
 $$\hat{f}(\Lambda):=\sum_{\bm{v}\in \Lambda_{\rm pr}}f(\bm{v}),$$
 where $\Lambda_{\rm pr}$ is the set of primitive vectors of $\Lambda$. Clearly $\hat{f}(\Lambda) = \#(\Lambda_{\rm pr}\,\cap\, {\cS})$ when $f$ is the indicator function of a subset ${\cS}$ of $\R^2$.
 
 Let us briefly describe the history of the problem. The Siegel transform was originally defined by Siegel \cite{Siegel1945} as the sum over all nonzero lattice point for unimodular lattices of any rank. In the same paper Siegel proved a Mean Value Theorem for the Siegel transform, which in the primitive set-up amounts to 
\begin{equation}\label{Siegel}
\int_X\hat{f}\left(\Lambda\right)\,d\mu\left(\Lambda\right) = \frac{1}{\zeta(2)}\int_{\R^2}f\left(\bm{x}\right)\,d\bm{x}.
\end{equation}
for any bounded compactly supported $f$ on $\R^2$.
Since then there has been much work extending his result to higher moments. For example, in \cite{Rogers1955} Rogers proved a series of higher moment formulas, which in particular includes a second moment formula for the Siegel transform defined on the space of unimodular lattices of rank greater than $2$. However, his result did not give a second moment formula on $X$ as in our setting. For this setting, Schmidt \cite{Schmidt1960} proved an upper bound for the second moment of the primitive Siegel transform of indicator functions on $\R^2$. His bound was later logarithmically improved by Randol \cite{Randol1970} for discs centered at the origin and by Athreya and Margulis \cite{AthreyaMargulis09} for general indicator functions building on Randol's bound. Athreya and Konstantoulas \cite{AthreyaKonstantoulas2016} obtained similar bounds on the space of general symplectic lattices for certain family of indicator functions.  Continuing \cite{AthreyaKonstantoulas2016}, Kelmer and the second-named author \cite{KelmerYu2018} proved a second moment formula on the space of symplectic lattices $Y_n:=\Sp(2n,\R)/\Sp(2n,\Z)$. In particular, when $n=1$ we have $Y_1=X$ and their formula also applies to our setting\footnote{See also \cite{Fairchild19} for moment formulas of the Siegel-Veech transform recently obtained by Fairchild.}.
%a general second moment formula $($explicitly stated in \cite[Theorem 1]{KelmerYu2018}$)$ for all bounded and compact supported functions. 
However, for 
%the special case of computing $\|\hat{f}_r\|_2$ 
our applications all these formulas are not  explicit enough.
%as \eqref{implicit} and fall short for our application. 
 
 \medskip
 
 We now state an explicit second moment formula which we use to derive Theorem \ref{thm:measure}.
% $$=\#(\Lambda_{\rm pr}\cap S_r)
\begin{thm}\label{thm:secondmoment}
For any $ r\geq 0$ let ${\cS_r}$ be the open square with vertices given by $(\pm e^{-r},\pm e^{-r})$, and let $f_r$ be the indicator function of ${\cS_r}$.  {Then we have}
\begin{equation}\label{implicit}
\|\hat{f}_r\|_2^2\,=\frac{8}{\zeta(2)}\left(e^{-2r}+\int_{\cD_r}\Big(\frac{e^{-r}}{x_1}+\frac{e^{-r}}{x_2}-\frac{1}{x_1x_2}\Big)\,dx_1dx_2\right),
\end{equation}
where $$\cD_r:=\big\{\bm{x}=(x_1,x_2)\in {\cS_r}\ \bigm| x_1>0, \ x_2>0,\  x_1+x_2>e^r\big\},$$
%$\hat{f}_r: X\to \R$ is the \textit{primitive Siegel transform} of $f_r$ defined such that
%$$\hat{f}_r(\Lambda):=\sum_{\bm{v}\in \Lambda_{\rm pr}}f_r(\bm{v})=\#(\Lambda_{\rm pr}\cap S_r)$$
%with $\Lambda_{\rm pr}$ the set of primitive vectors of $\Lambda$, 
and $\|\cdot\|_2$ stands for the $L^2$-norm %$\|\hat{f}_r\|_2$ 
with respect to $\mu$.
\end{thm}

\begin{rmk}
When $r\geq \frac12\log 2$ the region $\cD_r$ is empty, and equation \eqref{implicit} simply reads as $\|\hat{f}_r\|_2^2\,=\frac{8e^{-2r}}{\zeta(2)}$. We note that the latter equality in fact already follows from Siegel's Mean Value Theorem, {since in this case for any unimodular lattice there can only be at most one pair of primitive lattice points allowed in ${\cS_r}$}, which implies that $\frac12\hat{f}_r$ is an indicator function on $X$. When $0\leq r<\frac12\log 2$, the region $\cD_r$ is not empty, and it is not hard to compute the integral in \eqref{implicit} explicitly, see \eqref{equ:explicit2} below. In particular, plugging $r=0$ into \eqref{implicit} we have $\|\hat{f}_0\|_2^2=\left(\frac{12}{\pi}\right)^2-8\approx 6.59$.
%We also note here that we will prove Theorem \ref{thm:secondmoment} by proving 
\end{rmk}

In \S\ref{2moment} we prove  a much more general second moment formula, see Theorem \ref{thm:secondmomentgeneral}, with an arbitrary bounded measurable subset $\cS$ of $\R^2$ in place of ${\cS_r}$. Theorem \ref{thm:secondmoment} is derived from Theorem \ref{thm:secondmomentgeneral} by taking $\cS = {\cS_r}$.

\subsection*{Acknowledgements}
The authors would like to thank Anurag Rao, Nick Wadleigh and Cheng Zheng  for many helpful conversations. Thanks are also due to the anonymous referee for a quick and careful report.

\section{The second moment formula}\label{2moment}
In this section, we %prove Theorem \ref{thm:secondmoment} and show how Theorem \ref{thm:secondmoment} implies an asymptotic measure formula for the level set $K_{r}$ when $r>0$ is small. In fact we will 
prove Theorem \ref{thm:secondmoment} by establishing the following second moment formula for quite general subsets of $\R^2$.
{\begin{thm}\label{thm:secondmomentgeneral}
Let ${{\mathcal{S}}}$ be a measurable bounded subset of $\R^2$, and let $f$ be the indicator function of ${{\mathcal{S}}}$. Let ${{\widetilde{\mathcal{S}}}}=\left\{\bm{x}\in \R^2\ \left|\ -\bm{x}\in {{\mathcal{S}}}\right.\right\}$. Then we have
%\begin{equation}\label{equ:general}
$$\|\hat{f}\|_2^2=\frac{1}{\zeta(2)}\left(\area({{\mathcal{S}}})+\area({{\mathcal{S}}}\cap {{\widetilde{\mathcal{S}}}})+\sum_{n\neq 0}\frac{\varphi(|n|)}{|n|}\int_{{{\mathcal{S}}}}\left|\cI_{\bm{x}}^n\right|d\bm{x}\right),$$
%\|\hat{f}\|_2^2=\frac{2}{\zeta(2)}\left({\area}({{\mathcal{S}}})+\sum_{n=1}^{\infty}\frac{\varphi(n)}{n}\int_{{\mathcal{S}}}\left|\cI_{\bm{x}}^n\right|d\bm{x}\right),
%\end{equation}
where $\varphi$ is the Euler's totient function, $\cI_{\bm{x}}^n\subset \R$ is defined by
$$\cI_{\bm{x}}^n:=\left\{t\in \R\ \left|\ n\left(\frac{-x_2}{x_1^2+x_2^2},\frac{x_1}{x_1^2+x_2^2}\right)+{t}(x_1,x_2)\in {{\mathcal{S}}}\right.\right\},$$
and $|\cI_{\bm{x}}^n|$ is the length of $\cI_{\bm{x}}^n$ with respect to the Lebesgue measure on $\R$.  
\end{thm}}

Before giving the proof let us make a few remarks about Theorem \ref{thm:secondmomentgeneral}. %First we note that one can extract from our proof a similar second moment formula for any bounded and measurable set, see Remark \ref{rmk:generalset}. Here we state it for symmetric sets only for the simplicity of the formula. 
First we note that for any bounded ${{\mathcal{S}}}$ there exists a sufficiently large $T > 0$ depending on ${{\mathcal{S}}}$ such that for any $|n|>T$ the set $\cI_{\bm{x}}^n$ is empty for all $\bm{x}\in {{\mathcal{S}}}$. Thus the series on the right hand side of \eqref{equ:general} is a finite sum. Next we note that if we further assume ${{\mathcal{S}}}$ is symmetric with respect to the origin, then by symmetry we have ${{\mathcal{S}}}\cap {{\widetilde{\mathcal{S}}}}={{\mathcal{S}}}$ and $|\cI_{\bm{x}}^n|=|\cI_{\bm{x}}^{-n}|$ for any $n\neq 0$. In particular, for such ${{\mathcal{S}}}$ we have the following slightly simpler formula
\begin{equation}\label{equ:general}
\|\hat{f}\|_2^2=\frac{2}{\zeta(2)}\left({\area}({{\mathcal{S}}})+\sum_{n=1}^{\infty}\frac{\varphi(n)}{n}\int_{{\mathcal{S}}}\left|\cI_{\bm{x}}^n\right|d\bm{x}\right).
\end{equation}
%Although possible in theory, computing these integrals is very difficult in general. However, for the following special case which is of our interest for application, we have a much more explicit formula.

Finally we note that for any $\Lambda\in X$ and $f$ as in Theorem \ref{thm:secondmomentgeneral} we have
$$\left(\hat{f}(\Lambda)\right)^2=\hat{f}(\Lambda)+\hat{\chi}_{{{\mathcal{S}}}\cap{{\widetilde{\mathcal{S}}}}}(\Lambda)+\mathop{\sum_{\bm{v}_1,\bm{v}_2\in \Lambda_{\rm pr}}}_{linearly\ independent}f(\bm{v}_1)f(\bm{v}_2).$$
Thus Theorem \ref{thm:secondmomentgeneral} together with \eqref{Siegel} %Siegel's Mean Value Theorem (see \eqref{equ:siegel} below) 
implies that
\begin{equation}\label{equ:rogers}
\int_X\mathop{\sum_{\bm{v}_1,\bm{v}_2\in \Lambda_{\rm pr}}}_{linearly\ independent}f(\bm{v}_1)f(\bm{v}_2)\,d\mu(\Lambda)=\frac{1}{\zeta(2)}\sum_{n\neq 0}\frac{\varphi(|n|)}{|n|}\int_{{\mathcal{S}}}\left|\cI_{\bm{x}}^n\right|\,d\bm{x}.
\end{equation}
It is worth pointing out that the above formula can be compared to its higher-dimensional analogue: when $f$ is an indicator function of a bounded measurable subset ${{\mathcal{S}}}$ of $\R^k$ with $k\geq 3$, $X = \SL_k(\R)/\SL_k(\Z)$ and $\mu$ is the Haar probability measure on $X$,
% When integrating the above summation on $\SL_k(\R)/\SL_k(\Z)$ for a bounded and compactly supported function $f$ defined on $\R^k$ with $k\geq 3$, 
according to Rogers' second moment formula \cite{Rogers1955} 
the left hand side of \eqref{equ:rogers} equals $\left(\frac{\vol({{\mathcal{S}}})}{\zeta(k)}\right)^2$. However, as we can see here the $k=2$ case is much more complicated, with the answer depending on both the shape and the position of ${{\mathcal{S}}}$.
\subsection{Coordinates and measures}
We fix coordinates on $G=\SL_2(\R)$ via the Iwasawa decomposition $G=KAN$ with
$$K=\left\{k_{\theta}\ \left|\ 0\leq \theta< 2\pi\right.\right\},\quad A=\left\{a_s\ \left|\ s\in \R\right.\right\}\ \textrm{and}\ N=\left\{u_t\ \left|\ t\in \R\right.\right\},$$
where $k_{\theta}=\begin{pmatrix}
\cos\theta & -\sin\theta\\
\sin\theta & \cos\theta\end{pmatrix}$, $a_s=\begin{pmatrix}
e^s & 0\\
0 & e^{-s}\end{pmatrix}$ and $u_t=\begin{pmatrix}
1 & t\\
0 & 1\end{pmatrix}$.
 Explicitly, under coordinates $g=k_{\theta}a_su_t$, $\mu$ is given by 
 \begin{equation}\label{Haarneasure}
d\mu(g)=\frac{1}{\zeta(2)} e^{2s}\,d\theta dsdt.
\end{equation}

%Using the standard fundamental domain
%$$\cF_{\G}=\{u_ta_yk_{\theta}\ |\ t^2+y^2> 1,\ |t|< \frac12,\ 0< \theta< \pi\}$$
%for $X$, we get\cP_{\hat{f}_r}
%\begin{equation}\label{omega}
%\omega=\frac{6}{\pi}.
%\end{equation}
There is a natural identification between the homogeneous space $G/N$ and $\R^2\ssm \{\bm{0}\}$ induced by the map $G\to \R^2\ssm \{\bm{0}\}$ sending $g=k_{\theta}a_su_t\in G$ to, \begin{equation}\label{tran}
\bm{x}(g)=\begin{pmatrix}
x_1\\
x_2\end{pmatrix}=g\begin{pmatrix}1\\
0\end{pmatrix}=\begin{pmatrix}
e^s\cos\theta\\
e^s\sin\theta\end{pmatrix} 
\end{equation}
the left column of $g$. The Lebesgue measure, $d\bm{x}$, on $\R^2\ssm \{\bm{0}\}\cong G/N$ can be expressed via the polar coordinates $(s,\theta)$ as
\begin{equation}\label{leb}
d\bm{x}(k_{\theta}a_s)=e^{2s}\,d\theta ds.
\end{equation}

\subsection{The second moment formula}
In this subsection we prove Theorem \ref{thm:secondmomentgeneral}, and with some more analysis we prove Theorem \ref{thm:secondmoment}. %{We first prove the following preliminary second moment identity using a standard unfolding argument. We note that this is a special case of \cite[Proposition 2.3]{KelmerYu2018} where the same identity was proved for the primitive Siegel transform restricted on the space of symplectic lattices. However, since the notation there is much more involved, we include a short proof here for readers' convenience.}
%We first recall a preliminary second moment formula using a standard unfolding argument which holds for any $\hat{f}$ with $f$ bounded and compactly supported on $\R^2$.
{As the first step of our computation we recall the following preliminary identity which relies on a standard unfolding argument. We note that one can find it in \cite[Chapter $\rom{8}$, $\S$1]{LangSL2R}, and we include a short proof here to make the paper self-contained. See also \cite[Proposition 2.3]{KelmerYu2018} for a generalization to the space of symplectic lattices.}
\begin{lem}%$($\cite[Proposition 2.3\footnote{The reader should notice that while here we work on the right coset space $G/\G$ rather than the left coset space $\G\bk G$ as in \cite{KelmerYu2018}, this formula follows from the exact same arguments.}]{KelmerYu2018}$)$
\label{unfolding}
For any bounded and compactly supported function $f$ on $\R^2$ {and for any bounded $F\in L^2(X,\mu)$ }we have 
$${\langle \hat{f}, F\rangle =\frac{1}{\zeta(2)}\int_{-\infty}^{\infty}\int_0^{2\pi}f(\bm{x}(k_{\theta}a_s))\overline{\cP_F(\bm{x}(k_{\theta}a_s))}e^{2s}\,d\theta ds},$$
where $\cP_F$ is defined by
$$\cP_F(\bm{x}(k_{\theta}a_s))=\int_0^1F(k_{\theta}a_su_t\Z^2)\,dt$$
with $k_{\theta}, a_s$ and $u_t$ as above, and $\langle, \rangle$ is the inner product on $L^2(X,\mu)$.
\end{lem}
\begin{proof}
Let $\G=\SL_2(\Z)$ and let $\G_{\infty}=\G\cap N$. Recall that there is an identification between $\G/ \G_{\infty}$ and $\Z^2_{\rm pr}$ sending $\gamma\G_{\infty}$ to $\gamma\begin{pmatrix}1\\
0\end{pmatrix}$. Using this identification, for any $\Lambda=g\Z^2$ with $g\in \SL_2(\R)$ we can write 
\begin{equation}\label{equ:prid}
\hat{f}(\Lambda)=\sum_{\bm{v}\in \Lambda_{\rm pr}}f(\bm{v})=\sum_{\bm{w}\in \Z^2_{\rm pr}}f(g\bm{w})=\sum_{\gamma\in\G/ \G_{\infty}}\tilde{f}(g\gamma),
\end{equation}
where $\tilde{f}(g):=f\left(g\begin{pmatrix}1\\
0\end{pmatrix}\right)$. We note that $\tilde{f}$ is a right $N$-invariant function on $G$. Let $\cF_{\G}$ be a fundamental domain for $X=G/\G$, and let $\cF_{\infty}$ be a fundamental domain for $G/\G_{\infty}$. Note that using the Iwasawa decomposition $G=KAN$ we can choose
\begin{equation}\label{equ:fund}
\cF_{\infty}=\left\{k_{\theta}a_su_t\ \left|\ 0< \theta< 2\pi,\ s\in \R,\ 0< t< 1\right.\right\}.
\end{equation}
Moreover, fix a set of coset representatives $\Sigma_{\infty}\subset \G$ for $\G/ \G_{\infty}$, and  note that $\bigcup_{\gamma\in\Sigma_{\infty}}\cF_{\G}\gamma$ is a disjoint union and  forms a fundamental domain for $G/\G_{\infty}$. %Thus we can choose the fundamental domain $\cF_{\G}$ and the coset representatives $\gamma\in G/\G_{\infty}$ such that $\bigcup_{\gamma\in\G/\G_{\infty}}\cF_{\G}\gamma=\cF_{\infty}$.
Now for any {bounded} $F\in L^2(X,\mu)$, using \eqref{Haarneasure}, \eqref{equ:prid}, \eqref{equ:fund} and the facts that $F$ is right $\G$-invariant and $\tilde{f}$ is right $N$-invariant, we have
\begin{align*}
\langle \hat{f}, F\rangle&:=\int_{\cF_{\G}}\hat{f}(g\Z^2)\overline{F(g\Z^2)}d\mu(g)=\sum_{\gamma\in\G/ \G_{\infty}}\int_{\cF_{\G}}\tilde{f}(g\gamma)\overline{F(g\Z^2)}\,d\mu(g)\\
&=\sum_{\gamma\in\Sigma_{\infty}}\int_{\cF_{\G}\gamma}\tilde{f}(g)\overline{F(g\Z^2)}d\mu(g)=\int_{\bigsqcup_{\gamma\in\Sigma_{\infty}}\cF_{\G}\gamma}\tilde{f}(g)\overline{F(g\Z^2)}\,d\mu(g)\\
&=\int_{\cF_{\infty}}\tilde{f}(g)\overline{F(g\Z^2)}d\mu(g)=\frac{1}{\zeta(2)}\int_{-\infty}^{\infty}\int_{0}^{2\pi}\int_0^1\tilde{f}(k_{\theta}a_su_t)\overline{F(k_{\theta}a_su_t\Z^2)}e^{2s}\,dtd\theta ds\\
&=\frac{1}{\zeta(2)}\int_{-\infty}^{\infty}\int_{0}^{2\pi}f(\bm{x}(k_{\theta}a_s))\int_0^1\overline{F(k_{\theta}a_su_t\Z^2)}\,dt\,e^{2s}\,d\theta ds.
\end{align*}
Finally, we note that the above equalities can be justified since $F$ is bounded and the defining series for $\hat{f}$ is absolutely convergent (see \cite[Lemma 16.10]{Veech1998}).
%In particular, taking $F=\hat{f}$ finishes the proof.
%Finally, using relation \eqref{leb} and noting that $\zeta(2)=\frac{\pi^2}{6}$, we get
%\begin{equation}\label{unfolding2}
%\|\hat{f}\|_2^2=\frac{1}{\zeta(2)}\int_{\R^2}\overline{f(\bm{x})}\cP_f(\bm{x})d\bm{x}.\qedhere
%\end{equation}
\end{proof}
%{\begin{rmk}\label{rmk:siegel}
%Taking $F=1\in L^2(X,\mu)$ in the above proof and using \eqref{leb} we get
%\begin{equation}\label{equ:siegel}
%\int_X\hat{f}(\Lambda)\,d\mu(\Lambda)=\frac{1}{\zeta(2)}\int_{-\infty}^{\infty}\int_{0}^{2\pi}f(\bm{x}(k_{\theta}a_s))e^{2s}\,d\theta ds=\frac{1}{\zeta(2)}\int_{\R^2}f(\bm{x})\,d\bm{x}.
%\end{equation}
%This gives a theorem  analogous %mean value theorem 
%to Siegel's Mean Value Theorem \cite{Siegel1945} {where the average is taken} for the {standard} Siegel transform summing over all nonzero lattice points. We note that \eqref{equ:siegel} in fact already follows from Siegel's Mean Value Theorem using the M\"{o}bius inversion formula and the fact that all nonzero lattice points can be written uniquely as  positive multiples of primitive lattice points. {\color{red}{(We already mentioned this mean value theorem \eqref{Siegel} in the introduction, so maybe delete this remark here? Or reference this remark there.)}}
%\end{rmk}}
With this {preliminary identity}, we can now give the 
\begin{proof}[Proof of Theorem \ref{thm:secondmomentgeneral}]
Using the relation \eqref{leb} and Lemma \ref{unfolding} we have
\begin{equation}\label{equ:pre}
\|\hat{f}\|_2^2=\frac{1}{\zeta(2)}\int_{\R^2}f(\bm{x}{(k_{\theta}a_s)})\cP_{\hat{f}}(\bm{x}(k_{\theta}a_s))d\bm{x}=\frac{1}{\zeta(2)}\int_{S}\cP_{\hat{f}}(\bm{x}(k_{\theta}a_s))\,d\bm{x},
\end{equation}
where
$$\cP_{\hat{f}}(\bm{x}(k_{\theta}a_s))=\int_0^1\hat{f}(k_{\theta}a_su_t\Z^2)\,dt$$
with $k_{\theta}, a_s$ and $u_t$ as before. First, by the definition of the {primitive} Siegel transform we have
$$\hat{f}(k_{\theta}a_su_t\Z^2)=\#\left\{(m,n)\in \Z^2_{\rm pr}\ \left|\ k_{\theta}a_su_t\begin{pmatrix}m\\n\end{pmatrix}\in {{\mathcal{S}}}\right.\right\}.$$
Thus for $\bm{x}(k_{\theta}a_s)\in {{\mathcal{S}}}$ and $0\leq t<1$ we have
$$\hat{f}(k_{\theta}a_su_t\Z^2)=\sum_{(m,n)\in \Z^2_{\rm pr}}\chi_{I^{(m,n)}_{\bm{x}(k_{\theta}a_s)}}(t),$$
where 
$$I^{(m,n)}_{\bm{x}(k_{\theta}a_s)}:=\left\{0\leq t<1\ \left|\ k_{\theta}a_su_t\begin{pmatrix}m\\n\end{pmatrix}\in {{\mathcal{S}}}\right.\right\},$$
implying that 
$$\cP_{\hat{f}}(\bm{x}(k_{\theta}a_s))=\sum_{(m,n)\in \Z^2_{\rm pr}}\left|I_{\bm{x}(k_{\theta}a_s)}^{(m,n)}\right|=\left|I_{\bm{x}(k_{\theta}a_s)}^{(1,0)}\right|+\left|I_{\bm{x}(k_{\theta}a_s)}^{(-1,0)}\right|+\mathop{\sum_{(m,n)\in \Z^2_{\rm pr}}}_{n\neq  0}\left|I_{\bm{x}(k_{\theta}a_s)}^{(m,n)}\right|.$$ 
%where for the second equality we used that $I^{(m,n)}_{\bm{x}(k_{\theta}a_s)}=I^{-(m,n)}_{\bm{x}(k_{\theta}a_s)}$ which follows from the fact that ${{\mathcal{S}}}$ is invariant under inversion.
Next, by direct computation we have for $\bm{x}(k_{\theta}a_s)=(x_1,x_2)=(e^s\cos\theta,e^s\sin\theta)\in {{\mathcal{S}}}$, 
%\begin{align*}
\begin{equation}\label{equ:keyequ}
k_{\theta}a_su_t\begin{pmatrix}m\\n\end{pmatrix}=n\begin{pmatrix}-e^{-s}\sin\theta\\ e^{-s}\cos\theta\end{pmatrix}+(m+nt)\begin{pmatrix}e^s\cos\theta\\e^s\sin\theta\end{pmatrix}
=n\begin{pmatrix}\frac{-x_2}{x_1^2+x_2^2}\\ \frac{x_1}{x_1^2+x_2^2}\end{pmatrix}+(m+nt)\begin{pmatrix}x_1\\x_2\end{pmatrix}.
\end{equation}
%\end{align*}
When $(m,n)=(1,0)$ we have for $\bm{x}(k_{\theta}a_s)\in {{\mathcal{S}}}$, $k_{\theta}a_su_t\begin{pmatrix} 1\\0\end{pmatrix}=\begin{pmatrix}x_1\\ x_2\end{pmatrix}$ is contained in ${{\mathcal{S}}}$ for any $0\leq t<1$. Thus $I_{\bm{x}(k_{\theta}a_s)}^{(1,0)}=[0,1)$ and $\left|I_{\bm{x}(k_{\theta}a_s)}^{(1,0)}\right|=1$ for any $\bm{x}(k_{\theta}a_s)\in {{\mathcal{S}}}$. Similarly, when $(m,n)=(-1,0)$ we have for $\bm{x}(k_{\theta}a_s)\in {{\mathcal{S}}}$, $k_{\theta}a_su_t\begin{pmatrix} -1\\0\end{pmatrix}=\begin{pmatrix}-x_1\\ -x_2\end{pmatrix}$ is contained in ${{\mathcal{S}}}$ if and only if $\bm{x}\in {{\mathcal{S}}}\cap {{\widetilde{\mathcal{S}}}}$ with ${{\widetilde{\mathcal{S}}}}$ as in the theorem, implying that $I_{\bm{x}(k_{\theta}a_s)}^{(-1,0)}=[0,1)$ whenever $\bm{x}\in {{\mathcal{S}}}\cap {{\widetilde{\mathcal{S}}}}$.

When $n\neq 0$ by \eqref{equ:keyequ} we have for any integer $m$ coprime to $n$
\begin{align*}
\left|I_{\bm{x}}^{{(m,n)}}\right|&=\left|\left\{0\leq t<1\ \left|\ n\left(\frac{-x_2}{x_1^2+x_2^2},\frac{x_1}{x_1^2+x_2^2}\right)+(m+nt)(x_1,x_2)\in {{\mathcal{S}}}\right.\right\}\right|\\
&=\left|\left\{\frac{m}{n}\leq t<1+\frac{m}{n}\ \left|\ n\left(\frac{-x_2}{x_1^2+x_2^2},\frac{x_1}{x_1^2+x_2^2}\right)+nt(x_1,x_2)\in {{\mathcal{S}}}\right.\right\}\right|.
\end{align*}
We note that as $m$ runs through all the integers in each congruence class in $(\Z/|n|\Z)^{\times}$, the intervals $[\frac{m}{n},1+\frac{m}{n})$ cover $\R$ exactly once. Thus for $n\neq 0$
%\begin{align*}
$$\mathop{\sum_{m\in \Z}}_{(m,n)=1}\left|I_{\bm{x}(k_{\theta}a_s)}^{(m,n)}\right|=\varphi(|n|)\left|\left\{t\in \R\ \left|\ n\left(\frac{-x_2}{x_1^2+x_2^2},\frac{x_1}{x_1^2+x_2^2}\right)+{nt}(x_1,x_2)\in {{\mathcal{S}}}\right.\right\}\right|
=\frac{\varphi(|n|)}{|n|}\left|\cI_{\bm{x}}^n\right|,$$
%\end{align*}
where $\varphi$ is the Euler's totient function and $\cI_{\bm{x}}^n$ is as in Theorem \ref{thm:secondmomentgeneral}. We thus have for $\bm{x}\in {{\mathcal{S}}}$
$$\cP_{\hat{f}}(\bm{x})=1+\chi_{{{\mathcal{S}}}\cap {{\widetilde{\mathcal{S}}}}}(\bm{x})+\sum_{n\neq 0}\frac{\varphi(|n|)}{|n|}\left|\cI_{\bm{x}}^n\right|.$$
We conclude the proof by plugging the above equation into \eqref{equ:pre}.
\end{proof}
We can now give the 
{\begin{proof}[Proof of Theorem \ref{thm:secondmoment}]
To simplify notation for any $\bm{x}\in \R^2$, $t\in \R$ and $n\geq 1$ let 
$$\bm{v}(\bm{x},t,n):=n\left(\frac{-x_2}{x_1^2+x_2^2},\frac{x_1}{x_1^2+x_2^2}\right)+t(x_1,x_2).$$
First we note that $\left\|\bm{v}(\bm{x},t,n)\right\|_2^2=\frac{n^2}{x_1^2+x_2^2}+t^2(x_1^2+x_2^2)\geq \frac{n^2}{x_1^2+x_2^2}$, where $\|\cdot\|_2$ stands for the standard Euclidean norm on $\R^2$. Thus for $\bm{x}\in {\cS_r}$ and $n\geq 2$ %then 
we have 
$$\left\|\bm{v}(\bm{x},t,n)\right\|\geq \frac{\sqrt{2}}{2}\left\|\bm{v}(\bm{x},t,n)\right\|_2\geq \frac{\sqrt{2}}{\|\bm{x}\|_2}>e^r\geq e^{-r}, $$
implying that $\cI_{\bm{x}}^n$ is empty for any $\bm{x}\in {\cS_r}$ and any $n\geq 2$. Here $\|\cdot\|$ stands for the supremum norm on $\R^2$, and for the third inequality we used the fact that $\|\bm{x}\|_2<\sqrt{2}e^{-r}$, which follows from $\bm{x}$ being an element of ${\cS_r}$. Since ${\cS_r}$ is symmetric with respect to the origin, applying \eqref{equ:general} to $f=f_r$ we get
\begin{equation}\label{equ:med}
\|\hat{f}_r\|_2^2=\frac{8e^{-2r}}{\zeta(2)}+\frac{2}{\zeta(2)}\int_{{\cS_r}}\left|\cI_{\bm{x}}^1\right|d\bm{x}=\frac{8e^{-2r}}{\zeta(2)}+\frac{8}{\zeta(2)}\int_{{\cS_r^+}}\left|\cI_{\bm{x}}^1\right|d\bm{x},
\end{equation}
where ${\cS_r^+}$ is the intersection of ${\cS_r}$ with the first quadrant, and for the second equality we used the fact that $\left|\cI_{(x_1,x_2)}^1\right|=\left|\cI_{(\pm x_1,\pm x_2)}^1\right|$ which follows from the invariance of ${\cS_r}$ under reflections around the coordinate axes.
%fact that if $(x_1,x_2)\in S_r$ then so do $(\pm x_1,\pm x_2)$. 
We note that for $\bm{x}\in {\cS_r^+}$, $(\frac{-x_2}{x_1^2+x_2^2},\frac{x_1}{x_1^2+x_2^2})+t(x_1,x_2)\in {\cS_r}$ if and only if 
$$-\frac{e^{-r}}{x_1}+\frac{x_2}{x_1(x_1^2+x_2^2)}<t<\frac{e^{-r}}{x_1}+\frac{x_2}{x_1(x_1^2+x_2^2)},$$
and
$$-\frac{e^{-r}}{x_2}-\frac{x_1}{x_2(x_1^2+x_2^2)}<t<\frac{e^{-r}}{x_2}-\frac{x_1}{x_2(x_1^2+x_2^2)}.$$
By direct computation if $r\geq \frac12\log 2$ then there is no $t\in \R$ satisfying above inequalities. Thus $\cI_{\bm{x}}^1$ is empty, and the integral in the right hand side of \eqref{equ:med} is zero. If $0\leq r<\frac12\log 2$, we define for any $\bm{x}\in {\cS_r^+}$,
$$L(\bm{x}):=\max\left\{-\frac{e^{-r}}{x_1}+\frac{x_2}{x_1(x_1^2+x_2^2)},-\frac{e^{-r}}{x_2}-\frac{x_1}{x_2(x_1^2+x_2^2)}\right\}$$
and 
$$U(\bm{x}):=\min\left\{\frac{e^{-r}}{x_1}+\frac{x_2}{x_1(x_1^2+x_2^2)},\frac{e^{-r}}{x_2}-\frac{x_1}{x_2(x_1^2+x_2^2)}\right\}.$$
It is not hard to verify that as long as $0\leq r<\frac12\log 2$, for $\bm{x}\in {\cS_r^+}$ we have
$$L(\bm{x})=-\frac{e^{-r}}{x_1}+\frac{x_2}{x_1(x_1^2+x_2^2)}\quad \textrm{and}\quad U(\bm{x})=\frac{e^{-r}}{x_2}-\frac{x_1}{x_2(x_1^2+x_2^2)}.$$
%\commdk{I think we should explain that the equalities you wrote in red hold because of the assumptions $r\le\frac12\log 2$ and $0\le x_1,x_2\le 1$.} 
Thus $\cI^1_{\bm{x}}$ is nonempty if and only if $L(\bm{x})<U(\bm{x})$ and whenever it is nonempty we have $\cI^1_{\bm{x}}=(-\frac{e^{-r}}{x_1}+\frac{x_2}{x_1(x_1^2+x_2^2)},\frac{e^{-r}}{x_2}-\frac{x_1}{x_2(x_1^2+x_2^2)})$ . By direct computation we have %for $0<r<\frac12\log 2$, 
$L(\bm{x})<U(\bm{x})$ if and only if $\bm{x}\in \cD_r=\left\{(x_1,x_2)\in {\cS_r^+}\ \left|\ x_1+x_2>e^r\right.\right\}$. %Moreover, when $\bm{x}\in \cD_r$ we have
%We note that $\cR_r$ is nonempty since $\frac{\sqrt{2}}{2}<r<1$. 
%$$L(\bm{x})=-\frac{e^{-r}}{x_1}+\frac{x_2}{x_1(x_1^2+x_2^2)}\quad \textrm{and}\quad U(\bm{x})=\frac{e^{-r}}{x_2}-\frac{x_1}{x_2(x_1^2+x_2^2)},$$
Hence
\begin{align*}
\|\hat{f}_r\|_2^2&=\frac{8e^{-2r}}{\zeta(2)}+\frac{8}{\zeta(2)}\int_{\cD_r}\left(\left(\frac{e^{-r}}{x_2}-\frac{x_1}{x_2(x_1^2+x_2^2)}\right)-\left(-\frac{e^{-r}}{x_1}+\frac{x_2}{x_1(x_1^2+x_2^2)}\right)\right)\,dx_1dx_2\\
&=\frac{8e^{-2r}}{\zeta(2)}+\frac{8}{\zeta(2)}\int_{\cD_r}\Big(\frac{e^{-r}}{x_1}+\frac{e^{-r}}{x_2}-\frac{1}{x_1x_2}\Big)\,dx_1dx_2.\qedhere
\end{align*}
\end{proof}}
%{\begin{rmk}\label{rmk:intro}
%When $r\geq \frac12\log 2$ then for any $\Lambda\in X$ there can only be at most one pair of primitive lattice points of $\pm\bm{v}\in \Lambda_{\rm pr}$ allowed in $S_r$, thus $\hat{f}_r=2\chi_{B_r^1}$, where $B_r^1=\left\{\Lambda\in X\ \left|\ \Lambda_{\rm pr}\cap . Thus \eqref{equ:siegel} implies that $\frac{4e^{-2r}}{\zeta(2)}=\mu(\hat{f}_r)=2\mu(B_r^1)$, and taking the second moment we get $\|\hat{f}_r\|_2^2=4\mu(B_r^1)=\frac{8e^{-2r}}{\zeta(2)}$.
%\end{rmk}}
Besides the sets ${\cS_r}$, another natural candidate to test formula \eqref{equ:general} is the family of indicator functions of {balls}. For any $R>0$ let ${{\mathcal{B}_R}}$ be the open ball of radius $R$ centered at the origin, and let $h_R$ be the indicator function of ${{\mathcal{B}_R}}$.
%the open disc centered at the origin with radius $R$. 
{We note that Randol \cite{Randol1970} established an asymptotic formula for $\|\hat{h}_R\|_2^2$ for large $R$, and here} we prove the following formula for $\|\hat{h}_R\|_2^2$: 
\begin{cor}\label{cor:disc}
For any $R>0$ let $h_R$ be as above. Then we have
$$\|\hat{h}_R\|_2^2=\frac{12 R^2}{\pi}+\frac{48}{\pi}\sum_{n=1}^{\left\lfloor{R^2}\right \rfloor}\varphi(n)\left(\frac{\sqrt{R^4-n^2}}{n}+\arcsin\left(\frac{n}{R^2}\right)-\frac{\pi}{2}\right).$$
\end{cor}
\begin{proof}
%Let ${{\mathcal{B}_R}}$ be the open disc {with radius $R$ centerеd at the origin}. 
Since ${{\mathcal{B}_R}}$ is symmetric with respect to the origin, we can apply \eqref{equ:general} to $\|\hat{h}_R\|_2^2$, and use $\zeta(2)=\frac{\pi^2}{6}$ to get
$$\|\hat{h}_R\|_2^2=\frac{12 R^2}{\pi} +\frac{12}{\pi^2}\sum_{n=1}^{\infty}\frac{\varphi(n)}{n}\int_{{{\mathcal{B}_R}}}|\cI_{\bm{x}}^n|\,d\bm{x},$$
where $$\cI_{\bm{x}}^n:=\left\{t\in\R\ \left|\ \left\|n\left(\frac{-x_2}{x_1^2+x_2^2},\frac{x_1}{x_1^2+x_2^2}\right)+{t}(x_1,x_2)\right\|_2\right.<R\right\}.$$ Using the polar coordinates, for any $(x_2,x_2)=(r\cos\theta,r\sin\theta)\in {{\mathcal{B}_R}}$ and $n\geq Rr$ we can write %(or equivalently $n> \left\lfloor{R^2}\right \rfloor$)
$$\left\|n\left(\frac{-x_2}{x_1^2+x_2^2},\frac{x_1}{x_1^2+x_2^2}\right)+{t}(x_1,x_2)\right\|_2^2=\frac{n^2}{r^2}+t^2r^2\geq R^2,$$
implying that $\cI_{\bm{x}}^n$ is empty whenever $n\geq Rr=R\|\bm{x}\|_2$. In particular, $\cI_{\bm{x}}^n$ is empty for any $\bm{x}\in {{\mathcal{B}_R}}$ if $n\geq R^2$. Similarly for any $1\leq n\leq \left\lfloor{R^2}\right \rfloor$, $\cI_{\bm{x}}^n$ is empty if $\|\bm{x}\|_2\leq \frac{n}{R}$, and $\cI_{\bm{x}}^n=\left(-\frac{\sqrt{R^2r^2-n^2}}{r^2},\frac{\sqrt{R^2r^2-n^2}}{r^2}\right)$ if $\frac{n}{R}<\|\bm{x}\|_2<R$. Hence %using polar coordinates we have
\begin{align*}
\|\hat{h}_R\|_2^2&=\frac{12 R^2}{\pi} +\frac{12}{\pi^2}\sum_{n=1}^{\left\lfloor{R^2}\right \rfloor}\frac{\varphi(n)}{n}{\int_0^{2\pi}}\int_{\frac{n}{R}}^R\frac{2\sqrt{R^2r^2-n^2}}{r^2}r\,dr{d\theta}\\
&=\frac{12 R^2}{\pi} +\frac{48}{\pi}\sum_{n=1}^{\left\lfloor{R^2}\right \rfloor}\varphi(n)\int_1^{\frac{R^2}{n}}\sqrt{1-r^{-2}}\,dr\\
&=\frac{12 R^2}{\pi}+\frac{48}{\pi}\sum_{n=1}^{\left\lfloor{R^2}\right \rfloor}\varphi(n)\left(\frac{\sqrt{R^4-n^2}}{n}+\arcsin\left(\frac{n}{R^2}\right)-\frac{\pi}{2}\right),
\end{align*}
where for the second equality we applied a change of variable $\frac{R}{n}r\mapsto r$, and for the last equality we used the fact  that $\int\sqrt{1-r^{-2}}dr=\sqrt{r^2-1}+\arcsin\left(\frac{1}{r}\right)+C$ for $r\geq 1$.
\end{proof}
\section{Measure estimates of the shrinking targets}\label{thickening}
In this section, using the methods developed in the previous section, we prove Theorem~\ref{thm:measure} and then use it to derive Theorem \ref{thm:thickening} and Corollary \ref{cor:series}.
%{As mentioned in the introduction, we first give an asymptotic measure formula for $K_{r}$ when $r>0$ is small. For this we have the following result which is a corollary of Theorem \ref{thm:secondmoment}.}
%\begin{cor}\label{cor:measure}
%For any $0<r<\frac12\log 2$ we have
%$$\mu\left(K_{r}\right)=\frac{4r^2\log\left(\frac{1}{r}\right)}{\zeta(2)}+O(r^2).$$
%\end{cor}
\begin{proof}[Proof of Theorem \ref{thm:measure}]
For any $r>0$, let $f_r$ be the indicator function of ${\cS_r}$ as before. For any integer $k\geq 0$, let $B_r^k\subset X$ be the set of unimodular lattices having $2k$ nonzero primitive points in ${\cS_r}$. First, we note that $K_{r}=B_r^0$ consists of lattices with no nonzero points in ${\cS_r}$. Moreover, for any $\Lambda\in X$, there are at most two linearly independent primitive points of $\Lambda$ inside ${\cS_r}$. 
%That is, 
%$$\#(\Lamfbda_{\rm pr}\cap S_r)\in \{0,2,4\}.$$
%where $\Lambda_{\rm pr}$ denotes the set of primitive lattice points of $\Lambda$. 
We thus have for any $r>0$
\begin{equation}\label{equ:me1}
\sum_{k=0}^2\mu(B_r^k)=1,
\end{equation}
and
$$\hat{f}_r=2\chi_{B_r^1}+4\chi_{B_r^2}.$$
{Thus we can take the first moment and apply \eqref{Siegel} to get}
\begin{equation}\label{equ:first}
\mu(B_r^1)+2\mu(B_r^2)=\frac12\int_X\hat{f}_r(\Lambda)d\mu(\Lambda)=\frac{2e^{-2r}}{\zeta(2)}.
\end{equation}
Taking the second moment of $\hat{f}_r$ we get
\begin{equation}\label{equ:second}
4\mu(B_r^1)+16\mu(B_r^2)=\|\hat{f}_r\|_2^2.
\end{equation}
Solving equations \eqref{equ:me1}, \eqref{equ:first} and \eqref{equ:second} and applying Theorem \ref{thm:secondmoment} to \eqref{equ:second}, we get
$$\mu(K_{r})=\mu(B_r^0)=1-\frac{2e^{-2r}}{\zeta(2)}+\frac{1}{\zeta(2)}\int_{\cD_r}\Big(\frac{e^{-r}}{x_1}+\frac{e^{-r}}{x_2}-\frac{1}{x_1x_2}\Big)\,dx_1dx_2.$$
By direct computation we have for $0<r<\frac12 \log 2$
\begin{equation}\label{equ:explicit}
\begin{aligned}
&\int_{\cD_r}\Big(\frac{e^{-r}}{x_1}+\frac{e^{-r}}{x_2}-\frac{1}{x_1x_2}\Big)\,dx_1dx_2\\
&=2(1-r)(2e^{-2r}-1+ r)+(2-2e^{-2r}-2r)\log (1-e^{-2r})-2r^2+\int_{1-e^{-2r}}^{e^{-2r}}\frac{\log t}{1-t}\,dt\\
%&=2(1-r)(2e^{-2r}-1+ r)+(2-2e^{-2r}-2r)\log (1-e^{-2r})-2r^2+\Li_2(1-t)\big|_{1-e^{-2r}}^{e^{-2r}}\\
&=2(1-r)(2e^{-2r}-1+ r)+(2-2e^{-2r}-2r)\log (1-e^{-2r})-2r^2+\Li_2(1-e^{-2r})-\Li_2(e^{-2r}),
%&=2-\zeta(2)-4r-4r^2\log r+O(r^2),
\end{aligned}
\end{equation}
where $\Li_s(z)=\sum_{k=1}^{\infty}\frac{z^k}{k^s}$ is the polylogarithm function. Now for the term $\log (1-e^{-2r})$, using the Taylor expansion $e^{-2r}=1-2r+2r^2+O(r^3)$, we get 
$$\log (1-e^{-2r})=\log(2r)+\log\left(1-r+O(r^2)\right)=\log(2r)-r+O(r^2).$$ 
Using the series representation $\Li_2(z)=\sum_{k=1}^{\infty}\frac{z^k}{k^2}$ we get that $\Li_2(1-e^{-2r})=2r+O(r^2)$. Finally for the term $\Li_2(e^{-2r})$ we have the expansion (see \cite[Equation (9.7)]{Wood1992})
$$\Li_2(e^{-2r})=-2r\big(1-\log(2r)\big)+\zeta(2)+O(r^2).$$ 
Plugging these %expansions 
into \eqref{equ:explicit} and using the expansion $e^{-2r}=1-2r+2r^2+O(r^3)$, we get
\begin{equation}\label{equ:explicit2}
\int_{\cD_r}\Big(\frac{e^{-r}}{x_1}+\frac{e^{-r}}{x_2}-\frac{1}{x_1x_2}\Big)\,dx_1dx_2=2-\zeta(2)-4r-4r^2\log r+O(r^2),
\end{equation}
implying that
\begin{align*}
\mu(K_{r})&=1-\frac{2e^{-2r}}{\zeta(2)}+\frac{1}{\zeta(2)}\left(2-\zeta(2)-4r-4r^2\log r+O(r^2)\right)\\
&=-\frac{4r^2\log r}{\zeta(2)}+O(r^2)
\end{align*}
finishing the proof.
\end{proof}

{To estimate the measure of the thickening, we will need the following two preliminary lemmas.} We note that by Hajos-Minkowski Theorem (see \cite[\rom{9}.1.3]{Cassels1997}) we have 
%{(Here we should have a reference and say that it is Hajos-Minkowski Theorem, maybe quoting Cassels' Geometry of Numbers?)}
$$K_{0}=\Delta^{-1}\{0\} = \bigcup_{x\in [0,1)}\begin{pmatrix}
1 & x\\
0 & 1\end{pmatrix}\Z^2\bigcup \begin{pmatrix}
1 & 0\\
x & 1\end{pmatrix}\Z^2.$$ 
A simple observation is that any $\Lambda\in K_{0}$ contains either the point $(1,0)$ or the point $(0,1)$. %Since when $r$ small, elements in ${K_{r}}$ are close to elements in ${K_{0}}$. 
Thus intuitively one shall expect that when $r$ is small, lattices in $K_{r}$ contain points close to either $(1,0)$ or $(0,1)$. For any $r>0$, let ${\cA_r}\subset \R^2$ be the closed rectangle with vertices $(\pm\sqrt{e^{2r}-1},e^r)$ and $(\pm\sqrt{e^{2r}-1},e^{-r})$ and let ${\cC_r}$ be the closed rectangle with vertices $(e^r,\pm\sqrt{e^{2r}-1})$ and $(e^{-r},\pm\sqrt{e^{2r}-1})$, see Figure \ref{fig:1} . The following lemma asserts that when $r$ is small, then any $\Lambda\in K_{r}$ contains points either in ${\cA_r}$ or {in} ${\cC_r}$ (noting that ${\cA_r}$ is a small rectangle containing $(0,1)$ and ${\cC_r}$ is a small rectangle containing $(1,0)$).   
\begin{figure}[H]
  \centering
   \begin{minipage}{0.45\textwidth}
        \centering
        \includegraphics[width=0.5\textwidth]{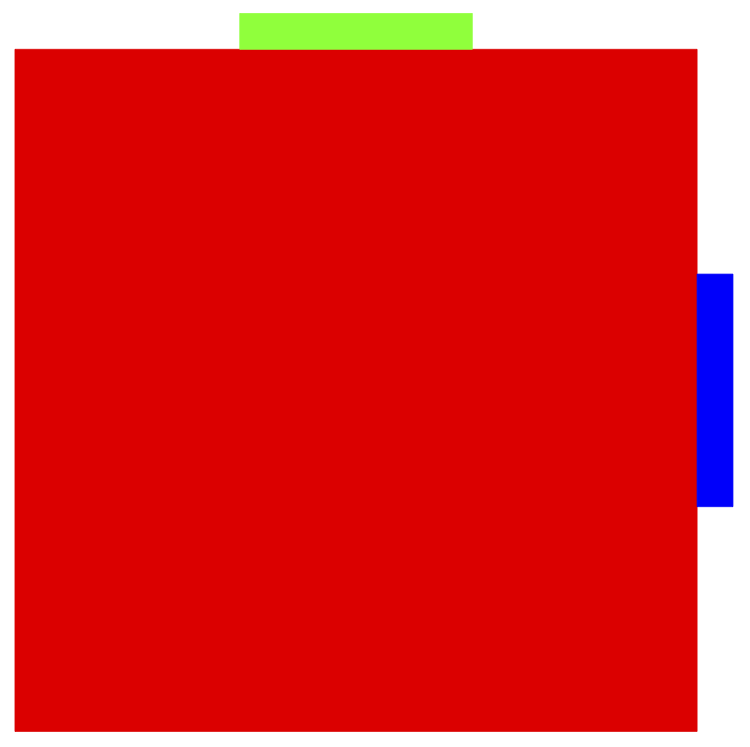} % second figure itself
        \caption{The square ${\cS_r}$ (red), the rectangles ${\cA_r}$ (green) and ${\cC_r}$ (blue).}\label{fig:1}
    \end{minipage} \hfill
     \begin{minipage}{0.45\textwidth}
        \centering
        \includegraphics[width=0.5\textwidth]{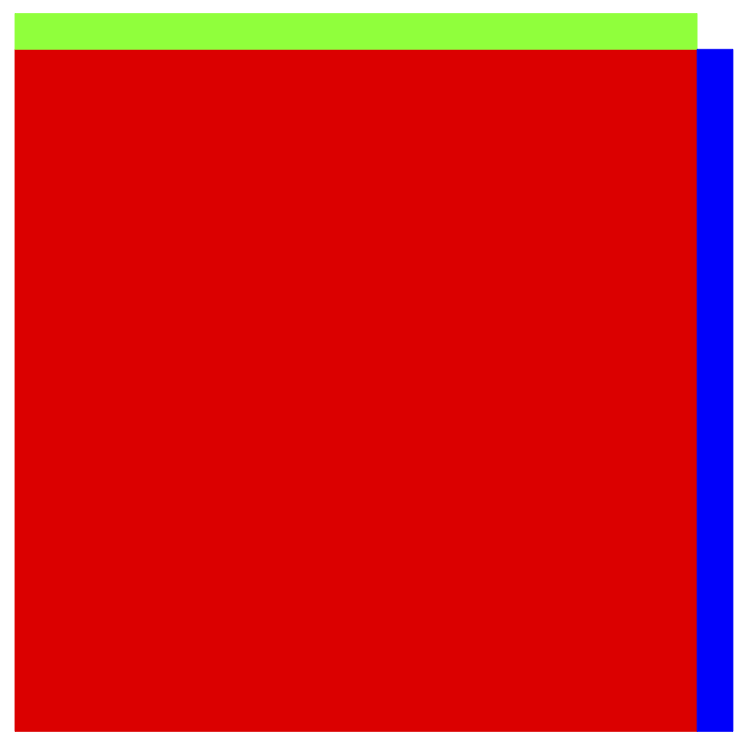} % first figure itself
        \caption{The square ${\cS_r}$ (red), the rectangles $\cU_r$ (green) and $\cR_r$ (blue).}\label{fig:2}
     \end{minipage}
%        \begin{minipage}{0.45\textwidth}
%        \centering
%        \includegraphics[width=0.6\textwidth]{"p3".png} % first figure itself
%        \caption{The square $S_r$ (red), the rectangles $\cU_r$ (green) and $\cR_r$ (blue).}\label{figure2}
%    \end{minipage}
%    \begin{minipage}{0.45\textwidth}
%        \centering
%        \includegraphics[width=0.6\textwidth]{"p4".png} % first figure itself
%        \caption{The square $S_r$ (red), the rectangles $\cU_r$ (green) and $\cR_r$ (blue).}\label{figure2}
%    \end{minipage}
% \centering
%  \includegraphics[width=0.2\textwidth]{"p1".png}
% \caption{The rectangle $a_sA_r$ (yellow) is contained in $S_r$ (red), and the rectangle $C_r$ (blue) is contained in $a_sS_r$ (grey); $r=0.006, s=0.3$.}\label{figure3}
% \end{figure}
% \begin{figure}[H]
% \centering
%  \includegraphics[width=0.2\textwidth]{"p2".png}
% \caption{a}\label{figure3}
 \end{figure}
%\begin{figure}[H]
%    \centering
%   \begin{minipage}{0.45\textwidth}
%        \centering
%        \includegraphics[width=0.7\textwidth]{"lemma 7b".png} % second figure itself
%        \caption{The square $S_r$ (red) and the rectangles $A_r$ (green) and $C_r$ (blue) with $r=0.006$.}\label{figure1}
%    \end{minipage} \hfill
%     \begin{minipage}{0.45\textwidth}
%        \centering
%        \includegraphics[width=0.7\textwidth]{"lemma 7a".png} % first figure itself
%        \caption{The square $S_r$ (red) and the rectangles $\cU_r$ (green) and $\cR_r$ (blue) with $r=0.006$.}\label{figure2}
%    \end{minipage}
%%     \begin{minipage}{0.45\textwidth}
%%        \centering
%%        \includegraphics[width=0.4\textwidth]{"lemma 8".png} % first figure itself
%%        \caption{The rectangle $a_sA_r$ (yellow) is contained in $S_r$ (red), and the rectangle $C_r$ (blue) is contained in $a_sS_r$ (grey); $r=0.006, s=0.3$.}\label{figure3}
%%    \end{minipage}
%\end{figure}
\begin{lem}\label{lem:nbhd}
Let ${\cA_r}$ and ${\cC_r}$ be as above. For any $0<r<\log 1.01$ and for any $\Lambda\in K_{r}$, we have $\Lambda_{\rm pr}\cap ({\cA_r}\cup {\cC_r})\neq \varnothing$. 
\end{lem}
\begin{proof}
Let $\cU_r$ be the closed rectangle with vertices $(\pm e^{-r}, e^{-r})$ and $(\pm e^{-r},e^r)$, and let $\cR_r$ be the closed rectangle with vertices $(e^{-r},\pm e^{-r})$ and $(e^r,\pm e^{-r})$, see Figure \ref{fig:2}. Let $${\widetilde{\cU}_r}:=\{\bm{x}\in \R^2\ |\ -\bm{x}\in \cU_r\}.$$ Consider the rectangle $\cU_r\sqcup {\cS_r}\sqcup \widetilde{\cU}_r$ and note that it has area $4$. For any $\varepsilon>0$ let $\cU_{r,\varepsilon}$ be the open rectangle with vertices $\left(\pm e^{-r},\pm(e^r+\varepsilon)\right)$. Applying the Minkowski's Convex Body Theorem to $\cU_{r,\varepsilon}$ and letting $\varepsilon$ approach zero, we see that for any $\Lambda\in X$, $\Lambda_{\rm{pr}}$ intersects $\cU_r\sqcup {\cS_r}\sqcup \widetilde{\cU}_r$ nontrivially. Now let $\Lambda\in K_{r}$; since $\Lambda$ has no nonzero point in ${\cS_r}$ and $\Lambda_{\rm{pr}}$ is invariant under inversion, we have $\Lambda_{\rm{pr}}\cap \cU_r\neq \varnothing$. Similarly we also have $\Lambda_{\rm{pr}}\cap \cR_r\neq \varnothing$. Moreover, {we note that for $0<r<\log 1.01$, we have $\Lambda\cap \cU_r=\Lambda_{\rm pr}\cap \cU_r$ and $\Lambda\cap \cR_r=\Lambda_{\rm pr}\cap \cR_r$. This is because otherwise there would be some nonzero point $\bm{v}\in \Lambda\cap (\cU_r\cup\cR_r)$ and some integer $k\geq 2$ such that $\frac{\bm{v}}{k}\in \Lambda_{\rm pr}$, but $\bm{v}\in \cU_r\cup\cR_r$ and $k\geq 2$ imply that $\frac{\bm{v}}{k}\in {\cS_r}$, contradicting the assumption that $\Lambda_{\rm pr}\cap {\cS_r}=\varnothing$.}
%it is easy to see that for $0<r<\log 1.01$, $\Lambda\cap \cU_r=\Lambda_{\rm pr}\cap \cU_r$ and $\Lambda\cap \cR_r=\Lambda_{\rm pr}\cap \cR_r$. 
Let $\bm{v}_1=(t_1,1+v_1)$ be a point in $\Lambda_{\rm pr}\cap \cU_r$ that is closest to the $y$-axis and let $\bm{v}_2=(1+v_2,t_2)$ be a point in $\Lambda_{\rm pr}\cap \cR_r$ that is closest to the $x$-axis. We thus have $|t_i|\leq e^{-r}$ and $e^{-r}\leq 1+v_i\leq e^r$ for $i=1,2$. 

Let $\cP_{\bm{v}_1,\bm{v}_2}$ be the parallelogram spanned by $\bm{v}_1$ and $\bm{v}_2$. Then we have for $0<r<\log 1.01$
$$|\cP_{\bm{v}_1,\bm{v}_2}|=|(1+v_1)(1+v_2)-t_1t_2|=(1+v_1)(1+v_2)-t_1t_2\leq e^{2r}+e^{-2r}<3,$$
where $|\cP_{\bm{v}_1,\bm{v}_2}|$ denote the area of $\cP_{\bm{v}_1,\bm{v}_2}$, and for the second equality we used that $$(1+v_1)(1+v_2)\geq e^{-2r}\geq |t_1t_2|.$$ Thus $|\cP_{\bm{v}_1,\bm{v}_2}|$ equals $1$ or $2$. We claim that $|\cP_{\bm{v}_1,\bm{v}_2}|=1$. Suppose not, then $|\cP_{\bm{v}_1,\bm{v}_2}|=2$ and we have for $0<r<\log 1.01$
$$t_1t_2=v_1+v_2+v_1v_2-1\leq 2(e^r-1)+(e^r-1)^2-1<0$$ 
and
$$|t_1t_2|=1-v_1-v_2-v_1v_2\geq 1-2(e^r-1)-(e^r-1)^2=2-e^{2r}>0.9.$$ 
This implies that $\min\{|t_1|,|t_2|\}>\frac{0.9}{e^{-r}}>0.9$. Since $t_1t_2<0$, without loss of generality we may assume that $t_2<0$. Then we have $-e^{-r}\leq t_2<-0.9$. On  one hand, since $|\cP_{\bm{v}_1,\bm{v}_2}|=2$ and $\bm{v}_1,\bm{v}_2\in \Lambda_{\rm pr}$, we have $$\bm{w}:=\frac{\bm{v}_1+\bm{v}_2}{2}=\left(\frac{t_1+1+v_2}{2},\frac{t_2+1+v_1}{2}\right)\in \Lambda.$$ On the other hand, we have $0<\frac{t_1+1+v_2}{2}\leq \frac{e^{-r}+e^r}{2}<e^{r}$, $0<\frac{t_2+1+v_1}{2}<\frac{1+v_1}{2}\leq \frac{e^{r}}{2}<e^{-r}$ and $\bm{w}\notin {\cS_r}$ implying that $\bm{w}\in \cR_r$. %Moreover, since $|\cP_{\bm{v}_1,\bm{v}_2}|=2$ and $\bm{v}_1,\bm{v}_2\in \Lambda_{\rm pr}$, we have $\bm{w}=\frac{\bm{v}_1+\bm{v}_2}{2}=(\frac{t_1+1+v_2}{2},\frac{t_2+1+v_1}{2})\in \Lambda$.  Thus we have $0<\frac{t_1+1+v_2}{2}\leq \frac{e^{-r}+e^r}{2}<e^{r}$ and $0<\frac{t_2+1+\e_1}{2}<\frac{1+\e_1}{2}\leq \frac{e^{r}}{2}<e^{-r}$. But since $\Lambda\ssm\{\bm{0}\}\cap S_r=\varnothing$, we have $\bm{w}\in \cR_r$. 
Thus $\bm{w}\in \Lambda\cap \cR_r=\Lambda_{\rm pr}\cap \cR_r$ is also a primitive vector of $\Lambda$. Moreover, since $-e^{-r}\leq t_2<-0.9$, we have $$0<\frac{t_2+1+v_1}{2}<\frac{e^r-0.9}{2}<\frac{1.01-0.9}{2}=0.055<|t_2|,$$ contradicting the assumption that $\bm{v}_2$ is the closest point in $\Lambda_{\rm pr}\cap \cR_r$ to the $x$-axis. We thus have proved the claim, and it implies that 
$$|t_1t_2|=|v_1+v_2+v_1v_2|\leq 2(e^r-1)+(e^r-1)^2=e^{2r}-1.$$
Hence we have $\min\{|t_1|,|t_2|\}\leq \sqrt{|t_1t_2|}\leq \sqrt{e^{2r}-1}$ which implies that $\Lambda_{\rm pr}\cap ({\cA_r}\cup {\cC_r})\neq \varnothing$ finishing the proof.
\end{proof}
The following lemma states that for $r>0$ small, the orbits $a_sK_{r}$ will completely leave the set $K_{r}$ very shortly, and will remain separated for quite a long time.
\begin{lem}\label{lem:trivialint}
For any $0<r<\log 1.01$ and any $6r\leq |s|\leq \log 1.9$, we have $$a_sK_{r}\cap K_{r}=\varnothing.$$
\end{lem}
\begin{proof}
 Suppose not, then there exists some $\Lambda\in a_sK_{r}\cap K_{r}$, and by definition the intersection of $\Lambda_{\rm pr}$ with ${\cS_r}\cup a_s{\cS_r}$ is empty. Without loss of generality we may assume that $s>0$. By Lemma \ref{lem:nbhd} we have $\Lambda_{\rm pr}\cap ({\cA_r}\cup {\cC_r})\neq \varnothing$ and similarly, $\Lambda_{\rm pr}\cap (a_s {\cA_r}\cup a_s {\cC_r})\neq \varnothing$. We note that $a_s {\cA_r}$ is the rectangle with vertices $(\pm e^s\sqrt{e^{2r}-1}, e^{r-s})$ and $(\pm e^s\sqrt{e^{2r}-1},e^{-r-s})$. Since $e^{6r}\leq e^s\leq 1.9$ we have $a_s{\cA_r}\subseteq {\cS_r}$ implying that $\Lambda_{\rm pr}\,\cap\, a_s {\cC_r}\neq \varnothing$. Similarly, we have ${\cC_r}\subseteq a_s {\cS_r}$ and this implies that $\Lambda_{\rm pr}\cap {\cA_r}\neq \varnothing$ (see Figure~\ref{fig:4}). Let $\bm{v}_1\in \Lambda_{\rm pr}\cap {\cA_r}$ and $\bm{v}_2\in \Lambda_{\rm pr}\cap a_s {\cC_r}$, and let $\cP_{\bm{v}_1,\bm{v}_2}$ be the parallelogram spanned by $\bm{v}_1$ and $\bm{v}_2$. Then for $0<r<\log 1.01$ and $6r\leq s\leq \log 1.9$ we have
$$1<e^{s-2r}-(e^{2r}-1)e^{-s}\leq |\cP_{\bm{v}_1,\bm{v}_2}|\leq e^{s+2r}+(e^{2r}-1)e^{-s}<2$$
{contradicting the fact that $|\cP_{\bm{v}_1,\bm{v}_2}|$ is a positive integer}.
%$e^{2r}<\frac{e^{2r}+\sqrt{5e^{4r}-4}}{2}<e^{t}<e^{-2r}+\sqrt{e^{-4r}+e^{-2r}-1}<2<\frac{e^{-r}}{\sqrt{e^{2r}-1}}$ we have $g_t A_r\subset g_tS_r$. Moreover, since $\Lambda\in g_tB_r$, $\Lambda_{\rm pr}\cap g_tS_r=\varnothing$. Hence $\Lambda_{\rm pr}\cap g_t C_r\neq \varnothing$. Similarly, $\Lambda_{\rm pr}\cap A_r\neq \varnothing$. 
\end{proof}
\begin{figure}[H]
  \centering
   \begin{minipage}{0.45\textwidth}
        \centering
        \includegraphics[width=0.7\textwidth]{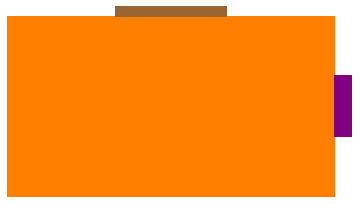}
        % first figure itself
        \caption{Figure \ref{fig:1} under the flow $a_s$: The rectangles $a_s{\cS_r}$ (orange), $a_s{\cA_r}$ (brown) and $a_s{\cC_r}$ (purple).}\label{fig:3}
     \end{minipage}
   \begin{minipage}{0.45\textwidth}
        \centering
        \includegraphics[width=0.7\textwidth]{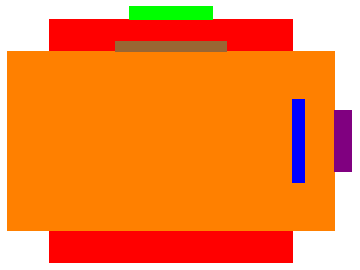} % second figure itself
        \caption{Figure \ref{fig:1} and Figure \ref{fig:3} in one picture: The rectangle $a_s{\cA_r}$ (brown) is contained in ${\cS_r}$ (red), the rectangle ${\cC_r}$ (blue) is contained in $a_s{\cS_r}$ (orange).}\label{fig:4}
    \end{minipage} \hfill
     \end{figure}
We can now give the 
\begin{proof}[Proof of Theorem \ref{thm:thickening}]
We prove the upper and lower bounds separately. For the upper bound, we first note that for any $\bm{v}\in \R^2$, $e^{-|s|}\|\bm{v}\|\leq \|a_s\bm{v}\|\leq e^{|s|}\|\bm{v}\|$. {Hence} for any $\Lambda\in X$ {we have}
$$|\Delta(a_s\Lambda)-\Delta(\Lambda)|\leq |s|.$$
{This implies that for any $s\in \R$ and any $r>0$
\begin{equation}\label{equ:upperestimate}
a_sK_{r}\subset K_{r+|s|}.
\end{equation}}
Let $N=\left \lceil{\frac{1}{r}}\right \rceil$. %Thus for any $0\leq t<\frac{1}{N}$ we have $g_{-t}B_r\subset B_{\frac{1}{N}+r}\subset B_{2r}$. 
{Using \eqref{equ:upperestimate} and the fact that $\frac{1}{N}\leq r$} we can estimate
$$\bigcup_{0\leq s<1}a_{-s}K_{r}=\bigcup_{0\leq i< N}\bigcup_{0\leq t<\frac{1}{N}}a_{-\frac{i}{N}}a_{-t}K_{r}\subset \bigcup_{0\leq i< N}a_{-\frac{i}{N}}K_{2r}.$$
Hence by Theorem \ref{thm:measure} and since $N\asymp \frac{1}{r}$ we have
$$\mu\left(\bigcup_{0\leq s<1}a_{-s}K_{r}\right)\leq %\mu\left(\bigcup_{0\leq i< N}a_{-\frac{i}{N}}\Delta^{-1}[0,2r]\right)\leq 
\sum_{i=0}^{N-1}\mu\left(a_{-\frac{i}{N}}K_{2r}\right)\asymp r\log\Big(\frac{1}{r}\Big).$$ 
For the lower bound, for $0<r<\log 1.01$ let $N=\left \lfloor{\frac{1}{6r}}\right \rfloor$. First we have
$$\bigcup_{0\leq i< \left \lfloor{N\log 1.9}\right \rfloor}a_{-\frac{i}{N}}K_{r}\subseteq \bigcup_{0\leq s<1}a_{-s}K_{r}.$$
Moreover, for each $0\leq i<j<\left \lfloor{N\log 1.9 }\right \rfloor$, $6r\leq \frac{1}{N}\leq \frac{j-i}{N}<\log 1.9$, thus by Lemma \ref{lem:trivialint} we have
$$a_{-\frac{i}{N}}K_{r}\cap a_{-\frac{j}{N}}K_{r}=a_{-\frac{j}{N}}\left(a_{\frac{j-i}{N}}K_{r} \cap K_{r}\right)=\varnothing.$$
Thus the union $\bigcup_{0\leq i< \left \lfloor{N\log 1.9}\right \rfloor}a_{-\frac{i}{N}}K_{r}$ is disjoint and, {again} applying Theorem \ref{thm:measure} {and noting that $N\asymp \frac{1}{r}$} we can estimate
\begin{displaymath}
\mu\left(\bigcup_{0\leq s<1}a_{-s}K_{r}\right)\geq %\mu\left(\bigcup_{0\leq i< \left \lfloor{\log 1.9 N}\right \rfloor}a_{-\frac{i}{N}}{K_{r}}\right)=
\sum_{i=0}^{\left \lfloor{N\log 1.9}\right \rfloor-1}\mu\left(a_{-\frac{i}{N}}K_{r}\right)\asymp r\log\Big(\frac{1}{r}\Big),
\end{displaymath}
finishing the proof.
\end{proof}

\begin{proof}[Proof of Corollary \ref{cor:series}]
First we note that we can assume that $\lim\limits_{s\to\infty}r(s)=0$ since otherwise both series would diverge. %Assuming $\lim\limits_{s\to\infty}r(s)=0$, then 
It follows that there exists %some constant 
$N>0$ such that for any $n>N$, $0<r(n)<\log 1.01$. Next, since $r(\cdot)$ is non-increasing,  for any $n>N$ we have
$$\bigcup_{0\leq s<1}a_{-s}K_{r(n+1)}\subset \widetilde{B}_n\subset \bigcup_{0\leq s<1}a_{-s}K_{r(n)}.$$
Moreover, since $n>N$ we have $0<r(n+1)\leq r(n)<\log 1.01$. Applying Theorem \ref{thm:thickening} to the left and right hand sides of the above inclusion relations we get
$$r(n+1)\log\left(\frac{1}{r(n+1)}\right)\ll
\mu\left(\widetilde{B}_n\right)\ll r(n)\log\left(\frac{1}{r(n)}\right)$$
which finishes the proof.
\end{proof}

\section{The dynamical Borel-Cantelli lemma}\label{sec:shrinkingtargets}
In this section we give the proof of Theorem \ref{thm:dynamical01} based on Theorem \ref{KW}. 
%As mentioned in the introduction we will use a number theoretic zero-one law obtained in \cite{KleinbockWadleigh2018}. 
%Let us 
%first briefly 
Recall 
%the main results from \cite{KleinbockWadleigh2018}. Recall 
that for a given 
%continuous, non-increasing 
function $\psi:[t_0,\infty)\to {(0,\infty)}$ with $t_0\geq 1$ fixed, we say a real number $x\in \R$ is \textit{$\psi$-Dirichlet} if the system of inequalities
$$|qx-p|<\psi(t)\quad \textrm{and}\quad |q|<t$$
has a solution in $(p,q)\in \Z\times (\Z\ssm\{0\})$ for all sufficiently large $t$. Let us denote by $D(\psi)$ the set of all $\psi$-Dirichlet numbers. 
Theorem \ref{KW} gives a zero-one law for the Lebesgue measure of $D(\psi)$ as follows:
if $\psi: [t_0,\infty)\to {(0,\infty)}$ is a continuous, non-increasing function satisfying 
\eqref{equ:con1psi} 
and \eqref{equ:con2psi}, then the series
\eqref{equ:zeroone} diverges (resp.\ converges) if and only if the Lebesgue measure of $D(\psi)$ $($resp. of $D(\psi)^c$$)$ is zero.

\medskip

For our purpose, we prove the following slightly modified version of \textsl{Dani Correspondence}.
\begin{lem}\label{lem:dani}
Let $\psi:[t_0,\infty)\to (0,\infty)$ be a continuous, non-increasing function satisfying \eqref{equ:con1psi} and \eqref{equ:con2psi}. Then there exists a unique continuous, non-increasing function 
$$r=r_{\psi}: [s_0,\infty)\to (0,\infty),\ \textrm{where $s_0=\frac12\log t_0-\frac12\log \psi(t_0)$},$$
such that 
%\begin{equation}\label{equ:con1}
%\textrm{$r(s)$ is positive and non-increasing,}
%\end{equation}
%and
\begin{equation}\label{equ:con2}
\textrm{the function $s\mapsto s+r(s)$ is non-decreasing,}
\end{equation}
and
\begin{equation}\label{equ:con3}
\psi(e^{s-r(s)})=e^{-s-r(s)}\ \textrm{for all $s\geq s_0$}.
\end{equation}
Conversely, given a continuous, non-increasing function $r: [s_0,\infty)\to (0,\infty)$ satisfying \eqref{equ:con2}, then there exists a unique continuous, non-increasing function $\psi=\psi_r: [t_0,\infty)\to (0,\infty)$ with $t_0=e^{s_0-r(s_0)}$ satisfying \eqref{equ:con1psi}, \eqref{equ:con2psi} and \eqref{equ:con3}. Furthermore, if we assume $\lim_{t\to\infty}t\psi(t)=1$ $($or equivalently, $\lim_{s\to\infty}r(s)=0$$)$, then the series in \eqref{equ:zeroone} diverges if and only if the series
\begin{equation}\label{equ:series}
\sum_nr(n)\log\left(\frac{1}{r(n)}\right)
\end{equation}
diverges.
\end{lem}
\begin{proof}
The correspondence between $\psi=\psi_r$ and $r=r_{\psi}$ follows from the exact same construction as in \cite[Lemma 8.3]{KleinbockMargulis1999}, where {$\psi(\cdot)$ and $r(\cdot)$} determine each other with the relations
$$e^s\psi(t)=e^{-r(s)}=e^{-s}t$$
with $s$ and $t$ satisfying $s=\frac12\log t-\frac12 \log \psi(t)$. The only difference is that here we require the two extra assumptions \eqref{equ:con1psi} and \eqref{equ:con2psi} on $\psi$ which are respectively equivalent to the assumptions that $r(\cdot)$ is non-increasing and $r(\cdot)$ is positive. We refer the reader to \cite[Lemma 8.3]{KleinbockMargulis1999} for more details about this correspondence.

For the furthermore part, first we {claim} that the series in \eqref{equ:zeroone} diverges if and only if the integral 
\begin{equation}\label{equ:integral}
\int_{t_0}^{\infty}\frac{-\left(1-t\psi(t)\right)\log\left(1-t\psi(t)\right)}{t}\,dt
\end{equation}
diverges. It suffices to show the function $G(t):=-\log\left(1-t\psi(t)\right)\left(1-t\psi(t)\right)$ is eventually non-increasing in $t$. Note that the function $T\mapsto -T\log T$ is strictly increasing on the interval $(0,e^{-1})$. Since $\lim_{t\to\infty}t\psi(t)=1$ and $t\psi(t)<1$ for all $t\geq t_0$, there exists some $T_0>t_0$ such that for all $t>T_0$, $0<1-t\psi(t)<e^{-1}$. Moreover, {together with} the assumption {\eqref{equ:con1psi}} we get that $G(t)$ is non-increasing in $t$ for any $T>T_0$, {finishing the proof the claim}.
%By assumptions on $\psi$ we know that $F(t)$  is non-increasing and $\lim_{t\to\infty}F(t)=0$. Thus there exists some constant $T>t_0$ such that $F(t)<e^{-1}$ for any $t>T$. Thus $G'(t)=-F'(t)\left(1+\log(F(t))\right)\leq 0$ implying that $G(t)$ is eventually non-increasing in $t$. 
Next, since $r(\cdot)$ is positive and non-increasing, we have $0<r(s)\leq r(s_0)$. Thus there exist constants $0<c_1<c_2$ such that for all $s\geq s_0$ and all $t\geq t_0$ with $s=\frac12\log t-\frac12 \log \psi(t)$, we have
$$c_1r(s)\leq 1-t\psi(t)=1-e^{-2r(s)}\leq c_2r(s).$$
{This also implies that
$$-\log(1-t\psi(t))=-\log\left(r(s)\right)+O_{c_1,c_2}(1)\asymp_{c_1,c_2}-\log(r(s)),$$
where for the second estimate we used that $\lim\limits_{s\to\infty}r(s)=0$.}
Moreover, since $r(\cdot)$ is non-increasing and continuous, it is differentiable at Lebesgue almost every $s\in \R$, and we denote by $r'(s)$ for its derivative at $s{\in\R}$ whenever it exists. Using the relation $t=e^{s-r(s)}$ we %have $dt=(1-r'(s))e^{s-r(s)}ds$, or equivalently, $\frac{dt}{t}=(1-r'(s))ds$. Thus we have
get  $\frac{dt}{t}=\big(1-r'(s)\big)\,ds$ for Lebesgue almost every $s\in \R$. We thus have
\begin{align*}
\int_{t_0}^{\infty}\frac{-\left(1-t\psi(t)\right)\log\left(1-t\psi(t)\right)}{t}\,dt&\asymp_{c_1,c_2} \int_{s_0}^{\infty}-r(s)\log(r(s))\big(1-r'(s)\big)\,ds\\
&\asymp \int_{s_0}^{\infty}-r(s)\log\big(r(s)\big)\,ds,
\end{align*}
where for the second estimate we used that $1\leq 1-r'(s)\leq 2$ for Lebesgue almost every $s\in \R$ which comes from the assumption \eqref{equ:con2} and that $r(\cdot)$ is non-increasing. Finally, we conclude the proof by noting that the integral $\int_{s_0}^{\infty}-r(s)\log\big(r(s)\big)\,ds$ diverges if and only if the series $\sum_{n}-r(n)\log(r(n))$ diverges since $\lim_{s\to\infty}r(s)=0$ and $r(\cdot)$ is non-increasing which imply that the function $s\mapsto -r(s)\log(r(s))$ is eventually non-increasing in $s$.
\end{proof}

%\begin{rmk}
%For later proof, we record here that $\psi=\psi_r$ and $r=r_{\psi}$ determines each other with the relations
%$$e^s\psi(t)=e^{-r(s)}=e^{-s}t,$$
%where $t$ and $s$ satisfies the relation $s=\frac12\log t-\frac12 \log \psi(t)$.
%\end{rmk}
As mentioned in the introduction, we have the following dynamical interpretation of $\psi$-Dirichlet numbers.
\begin{lem}$($\cite[Proposition 4.5]{KleinbockWadleigh2018}$)$\label{lem:hitting}
Let $\psi: [t_0,\infty)\to {(0,\infty)}$ be a continuous and non-increasing function satisfying \eqref{equ:con1psi} and \eqref{equ:con2psi}. Let $r=r_{\psi}$ be as in Lemma \ref{lem:dani}. Then $x\in D(\psi)^c$ if and only if 
\begin{equation}\label{equ:unbdd}
{a_s\Lambda_x\in K_{r(s)}\text{ 
for {an unbounded set of }s,}}
\end{equation} where $a_s=\diag (e^s,e^{-s})$ and $\Lambda_x=\begin{pmatrix} 1 & x\\
0 &1\end{pmatrix}\Z^2\in X$ are as before.
\end{lem}

Combining Theorem \ref{KW} with Lemma{s} \ref{lem:dani} and \ref{lem:hitting}, we immediately have the following zero-one law.
\begin{prop}\label{prop:zeroone}
Let $r :[s_0,\infty)\to (0,\infty)$ be continuous, non-increasing, {satisfying \eqref{equ:con2}} and such that $\lim\limits_{s\to\infty}r(s)=0$. Then \eqref{equ:unbdd} holds for Lebesgue almost every (resp.\ almost no) $x\in \R$ %the events 
%one has $a_s\Lambda_x\in K_{r(s)}$ for {an unbounded set of $s$} 
provided that the series %$\sum_nr(n)\log\left(\frac{1}{r(n)}\right)$ 
\eqref{equ:series} diverges (resp.\ converges).
\end{prop}

To connect the above proposition with the corresponding property of almost every $\Lambda\in X$, we need an auxiliary lemma, which borrows some ideas from the work \cite{KleinbockRao2019} of the first-named author with Anurag Rao.

\begin{lem}\label{lem:trick}
Let $r(\cdot)$ be as in Proposition \ref{prop:zeroone}. For any $c\in \R$ and $\lambda> 0$ let $$r_{c,\lambda}(s):=r(s+c)-\lambda e^{-2(s+c)},$$ and define
%$$\cD:=\{x\in\R\ |\ \textrm{$\forall c\in \R$ and $\lambda\geq 0$ the events $a_s\Lambda_x\in {K_{r_{c,\lambda}(s)}}$ happen for unbounded values of $s$}\}.$$
$${D_{c,\lambda}}:=\left\{x\in \R\ \left|\ \textrm{ $a_s\Lambda_x\in K_{r_{c,\lambda}(s)}$
for {an unbounded set of $s$}}\right.\right\}.$$
If the series %$\sum_nr(n)\log\left(\frac{1}{r(n)}\right)$ 
\eqref{equ:series}  diverges, then the set 
$${D:=\bigcap_{c\in\R}\bigcap_{\lambda> 0}D_{c,\lambda}}$$ 
has full Lebesgue measure.
\end{lem}
{\begin{rmk}
We note that by our assumption $r_{c,\lambda}(\cdot)$ is not necessarily always positive, and %here 
the set $K_{r_{c,\lambda}(s)}$ is empty whenever $r_{c,\lambda}(s)$ is negative.
\end{rmk}}
\begin{proof}[Proof of Lemma \ref{lem:trick}]
For any function $f: [s_f,\infty)\to (0,\infty)$ with $s_f\geq 1$ we denote
$$A_{\infty,f}:=\left\{x\in \R\ \left|\ \textrm{%the events 
$a_s\Lambda_x\in {K_{f(s)}}$ for {an unbounded set of $s>s_f$}}\right.\right\}$$
and $N_f:=\sum_{n\geq s_f}f(n)\log\left(\frac{1}{f(n)}\right)$. First we note that the divergence of the series $N_r$ is equivalent to the divergence of the series $N_{\frac12 r_c}$ for any $c\in \R$, where $r_c(s):=r(s+c) = r_{c,0}{(s)}$. Moreover, it is clear that $\frac12 r_c(\cdot)$ satisfies the assumptions in Proposition \ref{prop:zeroone}. Thus, by Proposition \ref{prop:zeroone}, if the series $N_r$ diverges, then the set $A_{\infty,\frac12r_c}$ is of full Lebesgue measure for any $c\in \R$. On the other hand, for any $c\in \R$ and $\lambda> 0$ let $f_{c,\lambda}(s)=\lambda e^{-2(s+c)}$. It is easy to check that $f_{c,\lambda}|_{[s_{c,\lambda},\infty)}$ satisfies the assumptions in Proposition \ref{prop:zeroone} with $s_{c,\lambda}:=\max\{\frac{\log(2\lambda)}{2}-c,1\}$, and the series $N_{f_{c,\lambda}}$ converges for any $c\in \R$ and $\lambda>0$. Thus by Proposition \ref{prop:zeroone} the set $A_{\infty,f_{c,\lambda}}$ is of zero Lebesgue measure for any $c\in \R$ and $\lambda>0$. Define 
$${\overline{A}}:=\bigcap_{c\in\R}A_{\infty,\frac12r_c}\quad \textrm{and}\quad {\underline{A}}:=\bigcup_{c\in\R}\bigcup_{\lambda>0}A_{\infty,f_{c,\lambda}}.$$
We note that since $r(\cdot)$ is non-increasing, for any $c_1<c_2$ we have $\frac12r_{c_1}\geq \frac12 r_{c_2}$ implying that $A_{\infty,\frac12r_{c_2}}\subset A_{\infty,\frac12r_{c_1}}$. Hence the family of sets $\{A_{\infty,\frac12r_{c}}\}_{c\in\R}$ is nested and ${\overline{A}}=\lim\limits_{c\to\infty}A_{\infty,\frac12r_c}$ is of full Lebesgue measure. Similarly, the family of sets $\{A_{\infty,f_{c,\lambda}}\}_{c\in\R,\lambda>0}$ is also nested and the set ${\underline{A}}=\lim\limits_{c\to-\infty}\lim\limits_{\lambda\to\infty}A_{\infty,f_{c,\lambda}}$ is of zero Lebesgue measure. Thus the set ${\overline{A}}\ssm {\underline{A}}$ is of full Lebesgue measure and it suffices to show that ${\overline{A}}\ssm {\underline{A}}\subset D$. That is, for any $x\in {\overline{A}}\ssm{\underline{A}}$ we want to show that for any $c\in \R$ and any $\lambda>0$ the events $a_s\Lambda_x\in K_{r_{c,\lambda}(s)}$ happen for {an unbounded set of $s$}. First we note that $x\in {\overline{A}}$ means that for any $c\in \R$ there exists an unbounded subset ${S_c}\subset\R$ such that $a_s\Lambda_x\in  {K_{\frac12r_c(s)}}$ for any $s\in {S_c}$. Secondly, we note that $x\notin {\underline{A}}$ means that for any $c\in \R$ and $\lambda>0$ there exists some constant $T_{c,\lambda}>0$ such that for any $s\geq T_{c,\lambda}$ we have $a_s\Lambda_x\in \Delta^{-1}(f_{c,\lambda}(s),\infty)$. In particular, for any $s\in {S_c}\cap (T_{c,\lambda},\infty)$ we have $$f_{c,\lambda}(s)<\Delta(a_s\Lambda_x)\leq \frac12r_c(s).$$ This implies that $0<\Delta(a_s\Lambda_x)\leq \frac12r_c(s)<\frac12r_c(s)+\frac12r_c(s)-f_{c,\lambda}(s)=r_{c,\lambda}(s)$ for any $s\in {S_c}\cap (T_{c,\lambda},\infty)$. Finally, we finish the proof by noting that since ${S_c}$ is unbounded, the set ${S_c}\cap (T_{c,\lambda},\infty)$ is also unbounded. 
\end{proof}
We can now give the 
\begin{proof}[Proof of Theorem \ref{thm:dynamical01}]
The convergent case follows directly from Corollary \ref{cor:series} and the classical Borel-Cantelli lemma, and we thus only need to prove the divergent case. Let $r: [s_0,\infty)\to (0,\infty)$ be continuous, non-increasing, {satisfying \eqref{equ:con2}} and such that the series %$\sum_nr(n)\log\left(\frac{1}{r(n)}\right)$
\eqref{equ:series} 
 diverges; we want to show that $\mu({B_{\infty}})=1$. First we note that we can assume that $\lim_{s\to\infty}r(s)=0$, since otherwise the result would follow from the ergodicity of the flow $\{a_s\}_{s > 0}$ on $X$. Let ${D:=\bigcap_{c\in\R}\bigcap_{\lambda> 0}D_{c,\lambda}}$ be as in Lemma \ref{lem:trick} and define ${B}\subset X$ such that
$${B}=\left\{\left.\begin{pmatrix}
a & 0\\
b & a^{-1}\end{pmatrix}\Lambda_x\in X\ \right|\ b\in \R,\  a>0,\  x\in {D}\right\} .$$
We note that by Lemma \ref{lem:trick} the set ${D}$ has full Lebesgue measure. Thus the set ${B}\subset X$ is also of full measure (with respect to $\mu$) and it suffices to show that ${B}\subset {B_{\infty}}$. First, by direct computation for $\Lambda=\begin{pmatrix}
a & 0\\
b & a^{-1}\end{pmatrix}\Lambda_x\in {B}$ we have
\begin{equation}\label{equ:reduction}
a_s\Lambda=\begin{pmatrix}
1 & 0\\
e^{-2s}a^{-1}b & 1\end{pmatrix}a_{s+\log a}\Lambda_x.
\end{equation}
%Since $x\in \cD$, $x$ is contained in $D(\psi_{c,\lambda})^c$ for any $c\in \R$ and $\lambda>0$. Hence by Lemma \ref{lem:hitting}, for any $c\in \R$ and $\lambda>0$, $a_s\Lambda_x\in {K_{r_{c,\lambda}(s)}}$ for unbounded values of $s$, or equivalently, $a_{s+\log a}\Lambda_x\in \Delta^{-1}[0,r_{c,\lambda}(s+\log a)]$ for unbounded values of $s$. 
Next, for any $y\in \R$ let $u^-_y=\begin{pmatrix}
1 & 0\\
y & 1\end{pmatrix}$. Note that for any $\bm{v}\in\R^2$, $\|u_y^-\bm{v}\|\leq (|y|+1)\|\bm{v}\|$. This implies that for any $\Lambda\in X$
%\begin{equation}\label{equ:unipotent}{
$$|\Delta(u^{-}_y\Lambda)-\Delta(\Lambda)|\leq \log(1+|y|).$$
%\end{equation}
%Hence for any $r>0$ and $y\in\R$ we have 
%$$u_y^{-}{K_{r}}\subset \Delta^{-1}[0,r+\log(1+|y|)].$$ 
Using the above inequality, the relation \eqref{equ:reduction} and the inequality $\log(1+x)<2x$ for all $x>0$, we get
%\begin{equation}\label{equ:reduction2}
$$\left|\Delta(a_s\Lambda)-\Delta(a_{s+\log a}\Lambda_x)\right|\leq 2a^{-1}|b|e^{-2s}.$$
%\end{equation}
Since $x\in {D}$ we have for any $c\in \R$ and any $\lambda>0$, $a_s\Lambda_x\in K_{r_{c,\lambda}(s)}$ for {an unbounded set} of $s$. In particular, taking $c=-\log a$, $\lambda=2a^{-1}|b|$ we get
$$0\leq \Delta(a_s\Lambda)\leq \Delta(a_{s-c}\Lambda_x)+\lambda e^{-2s}\leq r_{c,\lambda}(s-c)+\lambda e^{-2s}=r(s)$$
%\begin{align*}
%a_s\Lambda&=u^-_{e^{-2s}a^{-1}b}a_{s+\log a}\Lambda_x\in \Delta^{-1}[0,r_{c,\lambda}(s+\log a)+\log(1+e^{-2s}a^{-1}|b|)]\\
%&\subseteq \Delta^{-1}[0,r_{c,\lambda}(s+\log a)+\lambda e^{-2s}]=\Delta^{-1}[0, r(s)]
%\end{align*}
for {an unbounded set} of $s$, finishing the proof.
\end{proof}

\bibliographystyle{alpha}
\bibliography{DKbibliog}

\end{document}